\documentclass[lettersize,onecolumn]{IEEEtran}
\usepackage{amsmath,amsfonts}
\usepackage{algorithm}
\usepackage{algpseudocode}
\usepackage{array}
\usepackage[caption=false,font=normalsize,labelfont=sf,textfont=sf]{subfig}
\usepackage{textcomp}
\usepackage{stfloats}
\usepackage{url}
\usepackage{verbatim}
\usepackage{graphicx}
\usepackage{cite}
\usepackage{amssymb}
\usepackage{amsthm}
\usepackage{hyperref}
\usepackage{todonotes}
\usepackage{bm}
\usepackage{cuted}

\usepackage{amsthm}

\newtheoremstyle{ieeenormal}
{0.5em}              % space above
{0.5em}              % space below
{\normalfont}        % body font: upright, not italic
{\parindent}         % indent amount
{\normalfont}        % theorem head font controlled below
{}                   % punctuation controlled below
{0.5em}              % space after theorem head
{\thmname{\textit{#1}}\thmnumber{\textit{ #2}}\thmnote{\textit{ (#3)}}:}

\theoremstyle{ieeenormal}
\newtheorem{thm}{Theorem}
\newtheorem{defn}{Definition}
\newtheorem{remark}{Remark}
\newtheorem{lem}{Lemma}

\newtheorem{example}{Example}

\makeatletter
\renewenvironment{proof}[1][Proof]{%
	\par
	\pushQED{\qed}%
	\normalfont
	\topsep6\p@\@plus6\p@\relax
	\trivlist
	\item[\hskip 1.5em\itshape #1\@addpunct{:}]%
	\ignorespaces
}{%
	\popQED
	\endtrivlist
	\@endpefalse
}
\makeatother

\begin{document}

\title{Spectral Graph Uncertainty Principles via the Graph Fractional Fourier Transform}

\author{Yu Zhang and Bing-Zhao Li$^{\ast}$,~\IEEEmembership{Member,~IEEE}
	\thanks{The current study was partially supported by the Natural Science Foundation of Beijing Municipality (No. 4242011), and National Natural Science Foundation of China (No. 62571042), \textit{(Corresponding author: Bing-Zhao Li.)}}% <-this % stops a space
	\thanks{Yu Zhang is with the School of Mathematics and Statistics, Beijing Institute of Technology, Beijing 102488, China (e-mail: zhangyu$\_$bit@outlook.com).}
	\thanks{Bing-Zhao Li is with the School of Mathematics and Statistics, Beijing Institute of Technology, Beijing 102488, China (e-mail: li$\_$bingzhao@bit.edu.cn).}}
%\author{IEEE Publication Technology,~\IEEEmembership{Staff,~IEEE,}
%        % <-this % stops a space
%\thanks{This paper was produced by the IEEE Publication Technology Group. They are in Piscataway, NJ.}% <-this % stops a space
%\thanks{Manuscript received October 26, 2023; revised December 8, 2023.}}

\markboth{Journal of \LaTeX\ Class Files,~Vol.~1, No.~2, December~2023}%
{Shell \MakeLowercase{\textit{et al.}}: A Sample Article Using IEEEtran.cls for IEEE Journals}

\IEEEpubid{0000--0000~\copyright~2023 IEEE}

\maketitle

\begin{abstract}
This paper develops a graph fractional uncertainty principle in the graph fractional Fourier transform (GFRFT) domain. We introduce localization operators in the vertex domain and the graph fractional spectral domain, and build an operator framework to characterize the joint localization of graph signals. A sandwiched joint localization operator is first constructed, whose largest eigenvalue quantifies the attainable simultaneous concentration in the two domains. Then, rotated localization operators and the numerical range are used to provide a geometric description of the admissible uncertainty region, together with a polygonal approximation method for its computation. Numerical examples show that the fractional order reshapes the vertex-graph fractional spectral localization trade-off, and enlarge or shrink the uncertainty region relative to the graph Fourier transform based case. These results generalize classical graph uncertainty principles to the GFRFT domain and provide a flexible tool for graph fractional signal analysis and graph-adapted localized representations.
\end{abstract}


\section{Introduction}
\IEEEPARstart{G}{raph} signal processing (GSP) has emerged as an effective framework for analyzing data defined on irregular and network-structured domains, such as sensor networks, social networks, transportation systems, biological networks, and point clouds \cite{GFTlaplace,GFTadjacency,Gintroduction,Goverview,Ghistory}. In GSP, the graph topology provides a discrete geometric structure, while graph signals assign data values to the vertices. By using graph shift operators, such as the adjacency matrix or the graph Laplacian, classical signal processing concepts, including Fourier analysis, filtering, sampling, and localized representations, can be extended to irregular domains \cite{GFTlaplace,GFTadjacency,Gintroduction}. Among these tools, the graph Fourier transform (GFT), defined through the eigendecomposition of a graph shift operator, plays a central role in representing graph signals in the spectral domain.

A fundamental topic in signal analysis is the Heisenberg uncertainty principle, which describes the intrinsic limitation that a signal cannot be arbitrarily localized in two coupled domains. In classical time-domain signal processing, this principle characterizes the trade-off between temporal and frequency localization \cite{UncertaintyC,UncertaintyM}. Similarly, uncertainty principles in GSP have been extensively studied to capture the trade-off between localization in the vertex and graph spectral domains. An early spectral graph uncertainty principle was established by defining vertex and spectral spreads and characterizing the corresponding uncertainty curve \cite{GspectralUC}. Since distance-based spreads may depend strongly on the choice of graph distance \cite{Resistance,Geodesic,Diffusion}, alternative formulations have also been developed. In particular, \cite{GUncertainty} proposed a uncertainty principle based on energy concentration, where localization is quantified by the proportion of signal energy contained in prescribed vertex and spectral subsets, inspired by the works on prolate spheroidal wave functions \cite{PSWfunctions1,PSWfunctions2}. Later, \cite{GshapesUC} introduced a flexible operator-theoretic and numerical-range framework for describing uncertainty regions associated with general vertex and spectral filters. More recently, graph uncertainty principles have been extended from the GFT to the graph linear canonical transform (GLCT) domain \cite{GLCTsampling}, leading to GLCT-based sampling theory. Extensions have also been developed in the joint time-vertex signal processing and the generalized graph signal processing frameworks \cite{JFTUC,GGSPUC}. These studies provide important theoretical foundations for localization, sampling, and representation of graph signals.

However, most existing graph uncertainty principles are formulated in the conventional GFT domain. Although the GFT is a natural extension of the classical Fourier transform to graphs, it provides only a fixed spectral representation determined by the chosen graph shift operator. In many applications, additional degrees of freedom in the spectral representation are needed to characterize the transition from the vertex domain to the spectral domain and to process graph signals with chirp-like characteristics \cite{GFRFT,GFRFTconvolution,GFRFTspectral}. Inspired by the classical fractional Fourier transform \cite{FRFT1,FRFT2,DFRFT}, the graph fractional Fourier transform (GFRFT) has been proposed to provide fractional spectral representations of graph signals \cite{GFRFT,GFRFTconvolution,GFRFTspectral}. The GFRFT introduces a transform order that continuously controls the spectral representation, with the GFT as a special case. Recent studies on the GFRFT have established rigorous definitions \cite{GFRFT_unified,GFRFTdirected}, including power-function-based formulations for arbitrary graphs and operator constructions based on hyper-differential equations. The GFRFT has been successfully applied to vertex-frequency analysis \cite{GFRFT_unified,GFRFTdirected,WGFRFT,JFRFT}, filtering \cite{GFRFTfiltering,JFRFTfilter,JFRFTwiener,HGFRFT}, sampling \cite{HGFRFT,GFRFTsampling,JFRFTsampling}, and reconstruction \cite{LSBreconstruction,GGSPreconstruction,ComplexKernel}, thereby further enriching the theoretical and practical scope of GSP.

\IEEEpubidadjcol

Despite these advances, uncertainty principles in the GFRFT domain remain insufficiently understood. Existing graph uncertainty principles mainly characterize the trade-off between vertex-domain localization and standard graph spectral localization \cite{GUncertainty,GshapesUC}. When the spectral representation is replaced by the graph fractional spectral domain, it becomes necessary to establish a unified formulation of localization in the vertex and graph fractional spectral domains, to characterize how the fractional order changes the geometry of admissible localization regions, and to determine whether classical uncertainty principles based on energy concentration and numerical range can be extended to the GFRFT framework. Addressing these issues is essential for understanding the theoretical structure of the GFRFT and for designing graph-adapted localized representations with tunable spectral characteristics.

In this paper, we develop a graph spectral uncertainty principle via the GFRFT. Inspired by \cite{GshapesUC}, we introduce localization operators in the vertex domain and the graph fractional spectral domain, and use them to quantify the joint localization of graph signals. The proposed framework is operator-based and therefore applies to both soft localization filters \cite{Optlocalized,Slepian,Erb,Fiedler} and energy-concentration operators induced by prescribed subsets \cite{GUncertainty,Spatio1,Spatio2}. Based on these localization operators, we derive two complementary characterizations of uncertainty in the GFRFT domain. The first is based on a sandwiched joint localization operator, whose largest eigenvalue determines the attainable joint concentration. The second is based on a rotated localization operator and the numerical range, which provides a geometric description of the admissible uncertainty region and naturally leads to a polygonal numerical approximation method.

The main contributions of this paper are summarized as follows:
\begin{itemize}
	\item We construct vertex and fractional spectral localization operators associated with the GFRFT, providing a unified framework for measuring graph signal localization in the vertex domain and the graph fractional spectral domain.
	\item We establish a GFRFT-based uncertainty principle using a sandwiched joint operator. This result develops graph fractional uncertainty principles based on energy concentration and characterizes the limitation of simultaneous vertex and fractional spectral localization.
	\item We provide a geometric uncertainty characterization for the GFRFT based on rotated operators and the numerical range. This approach describes the admissible localization region through supporting lines and connects its boundary points with eigenvectors of a family of Hermitian operators.
	\item We develop a numerical approximation scheme for computing admissible uncertainty regions and present examples showing how the fractional order and the choice of localization filters affect the geometry of the vertex-fractional spectral trade-off.
\end{itemize}

The remainder of this paper is organized as follows. Section \ref{sec2} reviews basic concepts of GSP, the GFRFT, and existing graph uncertainty principles. Section \ref{sec3} introduces the vertex and graph fractional spectral localization operators and derives the corresponding joint localization measures. Section \ref{sec4} establishes the proposed GFRFT-based uncertainty principles through sandwiched and rotated operators. Section \ref{sec5} presents numerical examples and visualizations of the corresponding uncertainty regions. Finally, Section \ref{sec6} concludes the paper. Fig. \ref{fig:overview} presents an overview of the proposed graph fractional uncertainty principle framework and its applications.

\begin{figure}[!t]
	\centering
	\includegraphics[width=0.9\linewidth]{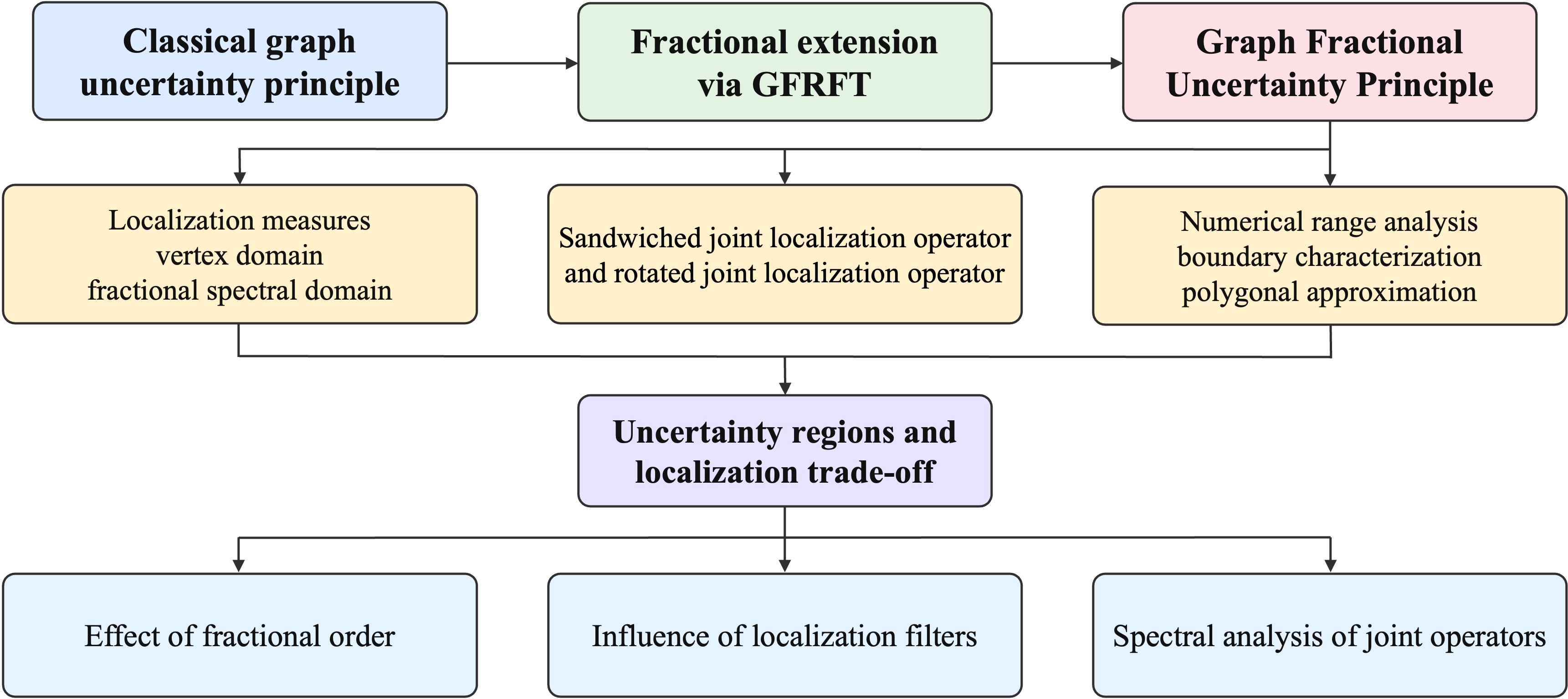}
	\vspace*{-8pt}
	\caption{Overview of the proposed graph fractional uncertainty principle framework and its applications.}
	\label{fig:overview}
\end{figure}

\section{Background}
\label{sec2}
This section aims to provide an overview of GSP and the GFRFT, as well as graph uncertainty principles, which are essential for constructing the graph fractional uncertainty principles.

\subsection{Graph and Graph Signals}
GSP extends classical signal processing tools to signals defined on irregular domains represented by graphs \cite{GraphTheory}. A graph is denoted by the triplet $\mathcal{G}=(\mathcal{V},\mathcal{E},\mathbf{W})$, where 
$\mathcal{V}=\{v_1,\ldots,v_N\}$ represents the set of vertices (or nodes), 
$\mathcal{E}\subseteq \mathcal{V}\times\mathcal{V}$ denotes the set of edges connecting pairs of vertices, 
and $\mathbf{W}\in\mathbb{R}^{N\times N}$ is the weighted adjacency matrix whose entries $w_{ij}$ represent the connection weights between vertices $v_i$ and $v_j$. 
In this work we focus on undirected graphs, for which $\mathbf{W}$ is symmetric and nonnegative.

A graph signal on $\mathcal{G}$ is defined as a function $x:\mathcal{V}\rightarrow\mathbb{R}$ (or $\mathbb{C}$) that assigns a scalar value to each vertex of the graph \cite{ComplexKernel}. 
Since the number of vertices is finite and the set $\mathcal{V}$ is ordered, the graph signal can be naturally represented as a vector 
\[
\bm{x}=[x_1,\ldots,x_N]^{\top} \in \mathbb{R}^N \; (\text{or } \mathbb{C}^N),
\]
where $x_i$ denotes the signal value associated with vertex $v_i$.

A fundamental operator in GSP is the graph Laplacian matrix defined as $\mathbf{L}=\mathbf{K}-\mathbf{W}$, where $\mathbf{K}=\mathrm{diag}(k_1, \ldots, k_N)$ is the diagonal degree matrix with entries $k_i = \sum^N_{j=1} w_{ij}$. The normalized graph Laplacian matrix is given by
\[
\bm{\mathcal{L}}=\mathbf{K}^{-1/2}\mathbf{L}\mathbf{K}^{-1/2}.
\]

Since the adjacency matrix $\mathbf{W}$ is symmetric for undirected graphs, the Laplacian operator $\boldsymbol{\mathcal{L}}$ is also symmetric and positive semidefinite. Therefore, it admits the eigenvalue decomposition
\begin{equation}
	\bm{\mathcal{L}}=\mathbf{U}\mathbf{\Xi }\mathbf{U}^{\mathrm{H}}, \label{L}
\end{equation}
where $\bm{\Xi}=\mathrm{diag}(\xi_1,\ldots,\xi_N)$ is the diagonal matrix of eigenvalues ordered as 
\[
0=\xi_1 \le \xi_2 \le \cdots \le \xi_N,
\]
and $\mathbf{U}=[\bm{u}_1,\ldots,\bm{u}_N]$ is the orthogonal matrix whose columns are the corresponding normalized eigenvectors.

\subsection{Graph Fractional Fourier Transform}
We briefly recall the GFT before introducing the GFRFT. Based on the eigendecomposition in Eq.~\eqref{L}, the GFT of a graph signal $\bm{x}$ is defined as \cite{GFTlaplace,GFTadjacency}
\begin{equation}
	\hat{\bm{x}}=\mathbf{U}^{\mathrm{H}}\bm{x},
\end{equation}
and the inverse GFT is defined as
\begin{equation}
	\bm{x}=\mathbf{U}\hat{\bm{x}}.
\end{equation}

The Laplacian eigenvalues $\xi_i$ represent graph frequencies, with smaller values corresponding to smoother signal variations. The Laplacian quadratic form can thus be written as
\begin{equation}
	\bm{x}^{\mathrm{H}}\bm{\mathcal{L}}\bm{x}
	=
	\hat{\bm{x}}^{\mathrm{H}}\bm{\Xi}\hat{\bm{x}}
	=
	\sum_{i=1}^{N}\xi_i|\hat{x}_i|^2, \label{xHLx}
\end{equation}
which characterizes the weighted spectral energy.

Extending the GFT, the GFRFT provides a fractional spectral representation analogous to the discrete fractional Fourier transform \cite{GFRFT}. For $\bm{x}\in\mathbb{C}^N$, the $\alpha$-th GFRFT is defined as
\begin{equation}
	\hat{\bm{x}}
	=
	\mathbf{F}^{\alpha}\bm{x}
	=
	\mathbf{Q}\mathbf{\Lambda}^{\alpha}\mathbf{Q}^{\mathrm{H}}\bm{x},
	\quad \alpha \in \mathbb{R},
	\label{GFRFT}
\end{equation}
with inverse
\begin{equation}
	\bm{x}
	=
	\mathbf{F}^{-\alpha}\hat{\bm{x}}
	=
	\mathbf{Q}\mathbf{\Lambda}^{-\alpha}\mathbf{Q}^{\mathrm{H}}\hat{\bm{x}}.
\end{equation}
Here, $\mathbf{Q}$ and $\mathbf{\Lambda}$ are obtained from
$\mathbf{U}^{\mathrm{H}}=\mathbf{Q}\mathbf{\Lambda}\mathbf{Q}^{\mathrm{H}}$, and $\mathbf{\Lambda}^{\alpha}$ is formed by raising each diagonal entry of $\mathbf{\Lambda}$ to the power $\alpha$. Eq.~\eqref{GFRFT} represents signals in the \emph{graph fractional spectral domain}. Let $\hat{\mathcal{G}}:=\{1,2,\ldots,N\}$ denote the index set associated with this domain. Based on the GFRFT, the graph fractional Laplacian matrix can be defined as \cite{GFRFT_unified}
\begin{equation}
	\bm{\mathcal{L}}^{\alpha}
	=
	\mathbf{F}^{-\alpha}\bm{\Xi}^{\alpha}\mathbf{F}^{\alpha}, \label{Lalpha}
\end{equation}
where $\bm{\Xi}^{\alpha} = \mathrm{diag}(\xi^{\alpha}_1,\ldots,\xi^{\alpha}_N)$, and this definition differs from traditional fractional Laplacians, as it is based on the GFRFT basis $\mathbf{F}^{\alpha}$ instead of the GFT basis $\mathbf{U}$, resulting in an $\alpha$-dependent spectral representation. Notably, $\alpha=0$ yields the identity transform, while $\alpha=1$ recovers the GFT. 

Moreover, the GFRFT-based convolution is defined as \cite{GFRFT,GFRFTconvolution}
\begin{equation}
	\bm{x} \ast \bm{y} := \mathbf{F}^{-\alpha}(\hat{\bm{x}} \circ \hat{\bm{y}}), \label{convolution}
\end{equation}
where $\circ$ denotes the Hadamard product.

\subsection{Graph Uncertainty Principles}
The classical Heisenberg uncertainty principle states that a signal cannot be simultaneously well localized in time and frequency \cite{UncertaintyC,UncertaintyM}, i.e.,
\begin{equation}
	\Delta_t^{2}\Delta_\omega^{2} \geq \frac{1}{4},
\end{equation}
where $\Delta_t^{2}$ and $\Delta_\omega^{2}$ denote temporal and spectral spreads.

For graph signals, an analogous uncertainty principle was established in \cite{GspectralUC}. Given $\bm{x} \in \mathbb{R}^N$, the vertex and spectral spreads are defined as
\begin{equation}
	\begin{cases}
		\Delta_{v}^{2}:=\min_{u_{0}\in \mathcal{V}} 
		\frac{1}{\|\bm{x}\|^{2}_2}
		\bm{x}^{\top}\mathbf{P}_{u_{0}}^{2}\bm{x},\\
		\Delta_s^{2}:=
		\frac{1}{\|\bm{x}\|^{2}_2}
		\sum_{i=1}^{N}\xi_{i}|\hat{x}_{i}|^{2},
	\end{cases}  \label{GFTspread}
\end{equation}
where $\mathbf{P}_{u_{0}}=\mathrm{diag}(d(u_{0},v_{1}),\ldots,d(u_{0},v_{N}))$. The achievable pairs form an uncertainty region whose lower boundary is given by
\[
\gamma(s)=\min_{\bm{x}} \ \Delta_v^2(\bm{x})
\quad \text{s.t. } \Delta_s^2(\bm{x})=s.
\]

This formulation relies on graph distances, which may be ill-defined for abstract graphs. To address this, an alternative energy concentration framework was proposed in \cite{GUncertainty}. Instead of spreads, localization is quantified by the energy ratios within prescribed subsets. For graph signals,
\begin{equation}
	\zeta
	=
	\frac{\sum_{i\in\mathcal{S}}|x_i|^2}{\|\bm{x}\|^2_2},
	\qquad
	\eta
	=
	\frac{\sum_{k\in\mathcal{F}}|\hat{x}_k|^2}{\|\bm{x}\|^2_2}, \label{GFTenergy}
\end{equation}
where $\mathcal{S}\subseteq\mathcal{V}$ and $\mathcal{F}\subseteq\hat{\mathcal{G}}$. The pair $(\zeta,\eta)$ characterizes joint vertex-spectral concentration and defines an alternative uncertainty principle \cite{GUncertainty}.

%\newpage
\section{Vertex and Spectral Localization in the GFRFT Domain}
\label{sec3}

To establish an uncertainty principle in the GFRFT domain, we first introduce a unified operator-based framework to quantify the localization of graph signals in both the vertex and graph fractional spectral domains. This framework is built upon two normalized filtering functions $\bm{f}, \bm{g} \in \mathbb{R}^N$ satisfying \cite{GshapesUC}
\[
0 \le f_i \le 1, \quad 0 \le \hat{g}_i \le 1, \quad 
\|\bm{f}\|_{\infty} = \|\hat{\bm{g}}\|_{\infty} = 1.
\]

\subsection{The GFRFT Problem Formulation}

Let $\bm{x}\in\mathbb{C}^{N}$ be a graph signal defined on $\mathcal{G}$. Without loss of generality, we assume $\|\bm{x}\|_2=1$, since the localization measures introduced below are scale-invariant.

Given the filters $\bm{f}$ and $\bm{g}$, the vertex domain and fractional spectral domain localization operators are defined as
\[
\begin{aligned}\mathbf{D}_{f} \bm{x} &:=\bm{f} \circ \bm{x} ,\\ \mathbf{B}^{\alpha }_{g} \bm{x} &:=\bm{g} \ast \bm{x} =\mathbf{F}^{-\alpha } (\hat{\bm{g} } \circ \hat{\bm{x} } )=\mathbf{F}^{-\alpha } \mathbf{D}_{\hat{g} } \mathbf{F}^{\alpha } \bm{x} ,\end{aligned} 
\]
where $\mathbf{D}_{f}=\mathrm{diag}(\bm{f})$, $\mathbf{D}_{\hat{g}}=\mathrm{diag}(\hat{\bm{g}})$, and the convolution is defined as in \eqref{convolution}. The following result characterizes the fundamental properties of these operators.

\begin{thm}
	The operators $\mathbf{D}_{f}$ and $\mathbf{B}^{\alpha}_{g}$ are Hermitian positive semidefinite, and their spectral norms are equal to $1$.
\end{thm}

\begin{proof}
	The proof is provided in Appendix \ref{appendixA}.
\end{proof}

Based on these operators, we define the corresponding quadratic forms
\begin{equation}
	\zeta_{f}(\bm{x}) :=\frac{\left< \mathbf{D}_{f}\bm{x},\bm{x}\right>  }{||\bm{x}||^{2}_2} =\frac{\bm{x}^{\mathrm{H}}\mathbf{D}_{f}\bm{x}}{||\bm{x}||^{2}_2}, \label{zeta}
\end{equation}
and
\begin{equation}
	\eta^{\alpha}_{g}(\bm{x}) :=\frac{\left< \mathbf{B}^{\alpha}_{g}\bm{x},\bm{x}\right>  }{||\bm{x}||^{2}_2} =\frac{\bm{x}^{\mathrm{H}}\mathbf{B}^{\alpha}_{g}\bm{x}}{||\bm{x}||^{2}_2}. \label{eta}
\end{equation}

A signal $\bm{x}$ is said to be perfectly localized in the vertex domain with respect to $\bm{f}$ if $\mathbf{D}_{f}\bm{x} = \bm{x}$, and perfectly localized in the GFRFT domain with respect to $\bm{g}$ if $\mathbf{B}^{\alpha}_{g}\bm{x} = \bm{x}$.
If both conditions hold simultaneously, $\bm{x}$ is perfectly localized in both domains.

The set of achievable localization pairs is defined as
\begin{equation}
	\Gamma(\mathbf{D}_{f},\mathbf{B}^{\alpha}_{g})
	:=
	\left\{
	(\zeta_f(\bm{x}), \eta^{\alpha}_g(\bm{x})) \;\middle|\; \|\bm{x}\|_2=1
	\right\}
	\subset [0,1]^2,
\end{equation}
which corresponds to the numerical range of the operator pair $(\mathbf{D}_{f},\mathbf{B}^{\alpha}_{g})$. The boundary of $\Gamma(\mathbf{D}_{f},\mathbf{B}^{\alpha}_{g})$ characterizes the proposed uncertainty principle in the GFRFT domain.

The fractional order $\alpha$ influences the geometry of the admissible region $\Gamma(\mathbf{D}_{f},\mathbf{B}^{\alpha}_{g})$ through the operator $\mathbf{B}^{\alpha}_{g}$, while the range $[0,1]^2$ remains invariant due to the normalization and positive semidefiniteness of the operators. In particular, different values of $\alpha$ induce different fractional spectral representations, leading to distinct trade-offs between vertex and spectral localization. 

\begin{remark}
	When $\alpha=1$, $\mathbf{B}^{\alpha}_{g}$ reduces to the classical GFT-based spectral localization operator, and $\Gamma$ coincides with the standard graph uncertainty region. When $\alpha=0$, $\mathbf{B}^{\alpha}_{g}$ becomes a vertex domain diagonal operator, yielding a purely vertex domain characterization.
\end{remark}

Therefore, $\alpha$ provides a continuous transition between different localization regimes, offering additional flexibility in controlling the vertex-spectral trade-off.

\subsection{Joint Operators of the GFRFT Domain}

To characterize the trade-off between vertex domain and graph fractional spectral domain localization, it is necessary to introduce operators that jointly capture their interaction \cite{GshapesUC}. While $\mathbf{D}_{f}$ and $\mathbf{B}^{\alpha}_{g}$ quantify localization in each domain separately, they do not directly describe their coupling. To this end, we construct joint operators that enable a unified analysis of vertex-spectral concentration and the admissible region $\Gamma$.

Given a vertex filter $\bm{f}$ and a fractional spectral filter $\bm{g}$, we define the joint operator
\begin{equation}
	\mathbf{S}^{\alpha}_{f,g}
	:=
	(\mathbf{B}^{\alpha}_{g})^{\frac{1}{2}}
	\mathbf{D}_{f}
	(\mathbf{B}^{\alpha}_{g})^{\frac{1}{2}},
\end{equation}
where $(\mathbf{B}^{\alpha}_{g})^{\frac{1}{2}}=\mathbf{F}^{-\alpha}\mathbf{D}_{\hat{g}}^{\frac{1}{2}}\mathbf{F}^{\alpha}$ denotes the Hermitian positive semidefinite square root of $\mathbf{B}^{\alpha}_{g}$.
\begin{thm} \label{thm2}
	The joint operator $\mathbf{S}^{\alpha}_{f,g} \in \mathbb{C}^{N\times N}$ is Hermitian positive semidefinite and satisfies
	\[
	\|\mathbf{S}^{\alpha}_{f,g}\|_2 \le 1.
	\]
\end{thm}
\begin{proof}
	The proof is provided in Appendix \ref{appendixB}.
\end{proof}

Based on $\mathbf{S}^{\alpha}_{f,g}$, we define the joint localization measure
\begin{equation}
	\varkappa^{\alpha}_{f,g}(\bm{x})
	:=
	\frac{\bm{x}^{\mathrm{H}}\mathbf{S}^{\alpha}_{f,g}\bm{x}}{\|\bm{x}\|_2^2}
	=
	\frac{\left((\mathbf{B}^{\alpha}_{g})^{\frac{1}{2}}\bm{x}\right)^{\mathrm{H}}
		\mathbf{D}_{f}
		\left((\mathbf{B}^{\alpha}_{g})^{\frac{1}{2}}\bm{x}\right)}{\|\bm{x}\|_2^2}. \label{varkappa}
\end{equation}

This quantity admits an equivalent representation in terms of the marginal localization measures defined in \eqref{zeta} and \eqref{eta}. Specifically, noting that
\[
\|(\mathbf{B}^{\alpha}_{g})^{\frac{1}{2}}\bm{x}\|_2^2
=
\bm{x}^{\mathrm{H}}\mathbf{B}^{\alpha}_{g}\bm{x}
=
\eta^{\alpha}_g(\bm{x}) \|\bm{x}\|_2^2,
\]
we obtain
\[
\varkappa^{\alpha}_{f,g}(\bm{x})
=
\zeta_f\!\big((\mathbf{B}^{\alpha}_{g})^{\frac{1}{2}}\bm{x}\big)\,
\eta^{\alpha}_g(\bm{x}).
\]
Therefore, $\varkappa^{\alpha}_{f,g}(\bm{x})$ can be interpreted as a joint localization measure, reflecting the simultaneous concentration of $\bm{x}$ in both the vertex and fractional spectral domains. In particular, when both $\zeta_f\!\big((\mathbf{B}^{\alpha}_{g})^{\frac{1}{2}}\bm{x}\big)$ and $\eta^{\alpha}_g(\bm{x})$ approach $1$, the joint localization $\varkappa^{\alpha}_{f,g}(\bm{x})$ also approaches $1$.

Since \textit{Theorem \ref{thm2}}, it admits the eigendecomposition, where the eigenvalues $\{\sigma_i\}^N_{i=1}$ with
\[
1 \ge \sigma_1 \ge \cdots \ge \sigma_N \ge 0.
\]
The largest eigenvalue satisfies $\sigma_{1} =  \|\mathbf{S}^{\alpha}_{f,g}\|_2$.

Moreover, by standard properties of operator norms, it follows that
\[
\sigma_{1} =\| (\mathbf{B}^{\alpha }_{g} )^{\frac{1}{2}}\mathbf{D}^{\frac{1}{2}}_{f} \|^{2}_{2} =\| \mathbf{D}^{\frac{1}{2}}_{f} (\mathbf{B}^{\alpha }_{g} )^{\frac{1}{2}}\|^{2}_{2} =\| \mathbf{D}^{\frac{1}{2}}_{f} \mathbf{B}^{\alpha }_{g} \mathbf{D}^{\frac{1}{2}}_{f} \|_{2},
\]
where $\mathbf{D}_{f}^{\frac{1}{2}}$ denotes the square root of $\mathbf{D}_f$. 

\begin{remark}
	In view of the above norm equivalence, the joint operator $\mathbf{S}^{\alpha}_{f,g}$ may be equivalently expressed as $\mathbf{D}^{\frac{1}{2}}_{f}\,\mathbf{B}^{\alpha }_{g}\,\mathbf{D}^{\frac{1}{2}}_{f}$, instead of $(\mathbf{B}^{\alpha}_{g})^{\frac{1}{2}}\mathbf{D}_{f}(\mathbf{B}^{\alpha}_{g})^{\frac{1}{2}}$.
\end{remark}

To further characterize the boundary of the admissible region, we consider a linear combination of the localization operators
\begin{equation}
	\mathbf{R}^{\alpha}_{f,g}(\beta)
	:=
	\cos\beta\, \mathbf{D}_f
	+
	\sin\beta\, \mathbf{B}^{\alpha}_{g},
	\quad \beta \in [0,2\pi). \label{R}
\end{equation}
Since $\mathbf{D}_f$ and $\mathbf{B}^{\alpha}_{g}$ are Hermitian, $\mathbf{R}^{\alpha}_{f,g}$ is Hermitian for any $\beta$. Moreover, when $\beta \in [0,\frac{\pi}{2}]$, both coefficients are nonnegative, and hence $\mathbf{R}^{\alpha}_{f,g}$ is positive semidefinite.

Using $\|\mathbf{D}_f\|_2 = \|\mathbf{B}^{\alpha}_{g}\|_2 = 1$, we have
\begin{equation}
	\|\mathbf{R}^{\alpha}_{f,g}\|_2
	\le
	|\cos\beta| + |\sin\beta|
	\le \sqrt{2}.
\end{equation}

Let $\{\rho_{i}(\beta)\}_{i=1}^N$ denote the eigenvalues of $\mathbf{R}^{\alpha}_{f,g}$. When $\beta \in [0,\frac{\pi}{2}]$, we have $0 \le \rho_{i}(\beta) \le \cos\beta + \sin\beta$. The corresponding quadratic form is given by
\begin{equation}
	\kappa^{\alpha}_{f,g}(\bm{x})
	:=
	\frac{\bm{x}^{\mathrm{H}}\mathbf{R}^{\alpha}_{f,g}\bm{x}}{\|\bm{x}\|_2^2}
	=
	\cos\beta\, \zeta_{f}(\bm{x})
	+
	\sin\beta\, \eta^{\alpha}_{g}(\bm{x}). \label{kappa}
\end{equation}
This expression can be interpreted as the projection of the localization pair $(\zeta_f(\bm{x}), \eta^{\alpha}_g(\bm{x}))$ onto the direction $(\cos\beta,\sin\beta)$. Hence, $\kappa^{\alpha}_{f,g}(\bm{x})$ serves as a directional measure of joint localization and plays the role of a supporting functional for the admissible region $\Gamma$.

In particular, when $\zeta_f(\bm{x})$ and $\eta^{\alpha}_g(\bm{x})$ are both close to $1$, we have
\[
\kappa^{\alpha}_{f,g}(\bm{x})
\approx
\cos\beta + \sin\beta,
\quad
\beta \in [0,\tfrac{\pi}{2}].
\]
Moreover, the special cases $\beta=0$ and $\beta=\frac{\pi}{2}$ reduce to the marginal measures $\zeta_f(\bm{x})$ and $\eta^{\alpha}_g(\bm{x})$, respectively.

\begin{remark}
	The fractional order $\alpha$ does not affect the spectral properties of the operators, but changes the underlying spectral representation. As a result, the admissible region $\Gamma$ remains within $[0,1]^2$, while its geometry and the associated uncertainty trade-off depend on $\alpha$.
\end{remark}

\subsection{Examples of Localization in the GFRFT Domain}

We now provide illustrative examples of vertex and fractional spectral filters by specifying different choices of $\bm{f}$ and $\bm{g}$.

\begin{example} \label{example1}
	(Distance-Based Localization) To promote localization around a reference vertex $u_0$, one may penalize signal components associated with vertices that are far from $u_0$. Let $d(u_0,v_i)$ denote a distance measure on the graph, e.g., the geodesic (shortest-path) distance, with $v_i \in \mathcal{V}$.
	
	The vertex filter $\bm{f}$ is defined as
	\begin{equation}
		f_i
		=
		1 - \left( \frac{d(u_0,v_i)}{\max_{v_i\in \mathcal{V}} d(u_0,v_i)} \right)^{a},
		\quad a > 0,  \label{distance}
	\end{equation}
	where the parameter $a$ controls the decay rate with respect to distance.
	
	For the fractional spectral filter $\bm{g}$, we consider a softened GFRFT bandlimited operator
	\begin{equation}
		\hat{g}_j
		=
		\begin{cases}
			\hat{g}^{b}(\omega_j), & \omega_j \in \mathcal{F}, \\
			0, & \text{otherwise},
		\end{cases} \label{hatgbwj}
	\end{equation}
	where $\mathcal{F} \subseteq \hat{\mathcal{G}}$ denotes a subset of the graph fractional spectral domain, and the weights satisfy $0 \le \hat{g}^{b}(\omega_j) \le 1$, typically decaying with increasing spectral index. Alternatively, the spectral weights $\hat{g}^{b}(\omega_j)$ can be defined as functions of the eigenvalues $\xi_j^{\alpha}$ of the graph fractional Laplacian operator. The parameter $b$ controls the smoothness of the spectral transition.
	
	The corresponding vertex localization measure becomes
	\begin{equation}
		\zeta_f(\bm{x})
		=
		1 -
		\frac{
			\bm{x}^{\mathrm{H}}\mathrm{diag}\!\big(d(u_0,v_i)^{a}\big)\bm{x}
		}{
			\max_{v_i\in \mathcal{V}} d(u_0,v_i)^{a} \,\|\bm{x}\|_2^2
		}.
	\end{equation}
	
This construction emphasizes signals concentrated around $u_0$ in the vertex domain, while restricting their energy within a prescribed fractional spectral subset $\mathcal{F}$. Similar distance-based filters have been widely used in continuous domains, such as intervals and spheres, through orthogonal expansions \cite{Optlocalized,Slepian,Erb,Fiedler}.
\end{example} 

\begin{example} \label{example2}
	(Parameterized Vertex-Spectral Localization) In the graph fractional spectral domain $\hat{\mathcal{G}}$, filters and localization operators can be defined directly via the spectral representation $\hat{\bm{x}}$. Specifically, we introduce the operators $\mathbf{D}_{\hat{f}}$ and $\mathbf{B}^{\alpha}_{\hat{g}}$ as
	\[
	\begin{aligned}\mathbf{D}_{\hat{f}}\hat{\bm{x}} & = \hat{\bm{f}} \circ \hat{\bm{x}},\\ 
		\mathbf{B}^{\alpha}_{\hat{g}}\hat{\bm{x}} &= \hat{\bm{g}} \ast \hat{\bm{x}} = \mathbf{F}^{-\alpha} \mathbf{D}_{\hat{\hat{g}}}\mathbf{F}^{\alpha} \hat{\bm{x}}.\end{aligned} 
	\]
	
	Because the spectrum of $\bm{\mathcal{L}}$ is contained in the interval $[0, 2]$, to ensure that the spectral responses $\hat{f}_i$ and $\hat{g}_j$ lie within $[0,1]$, they are defined as \cite{GshapesUC}
	\begin{equation}
		\hat{f}_i = 1 - \frac{1}{2}\xi^{\alpha}_i, \quad
		\hat{g}_j = 1 - \frac{1}{2}\xi^{\alpha}_j,
	\end{equation}
	respectively, and $\xi^{\alpha}$ is defined in Eq.~\eqref{Lalpha}.
\end{example}

\begin{example} \label{example3}
	(Spread-Based Localization) Consider the graph fractional spectral domain $\hat{\mathcal{G}}$ and define the spectral filter $\{\hat{g}_j\}_{j=1}^N$ as $\hat{g}_j = 1 - \frac{1}{2}\xi^{\alpha}_j$, where $\xi^{\alpha}_j$ denotes the $j$-th eigenvalue of the graph fractional Laplacian operator $\bm{\mathcal{L}}^{\alpha}$. The corresponding localization operator is given by
	\[
	\mathbf{B}^{\alpha}_{g}\bm{x}
	=
	\mathbf{F}^{-\alpha}\!\left(\mathbf{I}-\frac{1}{2}\bm{\Xi}^{\alpha}\right)\mathbf{F}^{\alpha} \bm{x}.
	\]
	
	For the vertex domain, we adopt the distance-based filter with $a=2$ in Eq.~\eqref{distance}, yielding
	\[
	\zeta_{f}(\bm{x})
	=
	1 -
	\frac{
		\bm{x}^{\mathrm{H}}\mathrm{diag}\!\big(d(u_0,v_i)^{2}\big)\bm{x}
	}{
		\max_{v_i\in \mathcal{V}} d(u_0,v_i)^{2} \,\|\bm{x}\|_2^2
	}.
	\]
	The corresponding fractional spectral localization measure is
	\[
	\eta^{\alpha}_{g}(\bm{x})
	=
	1 - \frac{\bm{x}^{\mathrm{H}}\bm{\mathcal{L}}^{\alpha}\bm{x}}{2\|\bm{x}\|_2^2}.
	\]
	In particular, when choosing $f_i=d(u_0,v_i)^2$ and $\hat{g}_j=\xi^{\alpha}_j$, the above formulation reduces to a spread-based characterization analogous to Eq.~\eqref{GFTspread} \cite{GspectralUC}. Specifically, 
	\[
	\zeta_{f}(\bm{x})
	=
	\frac{
		\bm{x}^{\mathrm{H}}\mathbf{P}_{u_0}\bm{x}
	}{\|\bm{x}\|_2^2}, \quad
	\eta^{\alpha}_{g}(\bm{x})
	=
	\frac{\bm{x}^{\mathrm{H}}\bm{\mathcal{L}}^{\alpha}\bm{x}}{\|\bm{x}\|_2^2},
	\]
	measure the vertex spread of $\bm{x}$ around $u_0$ and the graph fractional spectral spread, respectively. This yields a natural extension of the classical spread-based uncertainty principle to the graph fractional spectral domain.
\end{example}

\begin{example} \label{example4}
	(Energy Concentration-Based Localization) Energy concentration-based uncertainty principles are among the most widely used formulations in GSP \cite{GUncertainty}, as defined in Eq.~\eqref{GFTenergy}. This framework can be naturally extended to the graph fractional spectral domain.
	
	Specifically, we define the vertex-domain and fractional spectral-domain filters as
	\[
	f(v_i) =
	\begin{cases}
		1, & v_i \in \mathcal{S}, \\
		0, & \text{otherwise},
	\end{cases}
	\quad
	\hat{g}(\omega_j) =
	\begin{cases}
		1, & \omega_j \in \mathcal{F}, \\
		0, & \text{otherwise},
	\end{cases}
	\]
	where $f(v_i)=f_i$, $\hat{g}(\omega_j)=\hat{g}_j$, with $\mathcal{S}\subseteq\mathcal{V}$ and $\mathcal{F}\subseteq\hat{\mathcal{G}}$. This corresponds to the ideal bandlimited case of Eq.~\eqref{hatgbwj} with unit weights. In this setting, the operators $\mathbf{D}_f$ and $\mathbf{B}^{\alpha}_g$ are not only Hermitian but also idempotent, satisfying
	\[
	\mathbf{D}_f = \mathbf{D}_f^2,
	\quad
	\mathbf{B}^{\alpha}_g = (\mathbf{B}^{\alpha}_g)^2.
	\]
	Consequently, the joint operator reduces to
	\[
	\mathbf{S}^{\alpha}_{f,g}
	=
	\mathbf{B}^{\alpha}_g \mathbf{D}_f \mathbf{B}^{\alpha}_g,
	\]
	which corresponds to a projection onto the intersection of vertex and fractional spectral subspaces.
	
	This formulation is closely related to the classical time-frequency concentration problem, extensively studied by \cite{Spatio1,Spatio2}, where optimal joint concentration is characterized via projection operators.
\end{example}

\section{Graph Fractional Uncertainty Principles}
\label{sec4}

In this section, we extend classical graph uncertainty principles to the GFRFT domain. By leveraging the localization operators introduced in the previous section, we establish fractional uncertainty curves and regions that characterize the fundamental trade-off between vertex localization and fractional spectral localization.

\subsection{Uncertainty Principles with Sandwiched Operators}

By extending the graph uncertainty principle from the GFT framework to the GFRFT setting, we aim to characterize the admissible region $\Gamma(\mathbf{D}_f, \mathbf{B}^{\alpha}_g)$ in closed form. Specifically, our goal is to establish a trade-off between $\zeta_f(\bm{x})$ and $\eta^{\alpha}_g(\bm{x})$ through the localization operators $\mathbf{D}_f$ and $\mathbf{B}^{\alpha}_g$, and to identify signals that attain all admissible pairs, thereby fully characterizing the uncertainty principle.

Under normalized filters $\bm{f}$ and $\bm{g}$, we first derive an uncertainty relation governed by the largest eigenvalue $\sigma_1$ of the joint localization operator $\mathbf{S}^{\alpha}_{f,g}$. This formulation is consistent with the energy concentration setting in \textit{Example \ref{example4}}, while remaining applicable to general localization operators $\mathbf{D}_f$ and $\mathbf{B}^{\alpha}_g$. The following lemma characterizes the achievable region.

\begin{lem}\label{lem1}
	Let $\bm{x} \in \mathbb{C}^N$ satisfy $\|\bm{x}\|_2 = 1$, and define $\zeta_f(\bm{x}) = \|\mathbf{D}_f \bm{x}\|_2^2$, and $\eta^{\alpha}_g(\bm{x}) = \|\mathbf{B}^{\alpha}_g \bm{x}\|_2^2$. Assume $\sigma_1 < 1$ and $\sigma_1 \leq \zeta_f(\bm{x}) \eta^{\alpha}_g(\bm{x})$. Then
	\begin{equation}
		\arccos \sqrt{\zeta_f(\bm{x})}
		+
		\arccos \sqrt{\eta^{\alpha}_g(\bm{x})}
		\geq
		\arccos \sqrt{\sigma_1}.
		\label{arccoszeta_eta}
	\end{equation}
	Equivalently, $\eta^{\alpha}_g(\bm{x})$ is upper bounded by
	\begin{equation}
		\eta^{\alpha}_g(\bm{x})
		\le
		\left(
		\sqrt{\zeta_f(\bm{x}) \sigma_1}
		+
		\sqrt{(1-\zeta_f(\bm{x}))(1-\sigma_1)}
		\right)^2.
		\label{etaalpha}
	\end{equation}
\end{lem}
\begin{proof}
	The proof is provided in Appendix~\ref{appendixC}.
\end{proof}

For $\zeta_f(\bm{x}) \in [\sigma_1, 1]$, it follows that
\[
\sigma_1
\leq
\frac{\sigma_1}{\zeta_f(\bm{x})}
\leq
1 - \zeta_f(\bm{x}) + \sigma_1
\leq
\gamma(\zeta_f(\bm{x}))
\leq
1,
\]
because of the curve \cite{GshapesUC}
\[\gamma(\zeta_f(\bm{x}))
=
\left( \sqrt{\zeta_f(\bm{x}) \sigma_1}
+
\sqrt{(1-\zeta_f(\bm{x}))(1-\sigma_1)}\right)^2.
\]

Consequently, within the square $[\sigma_1,1]^2$, the following inclusions hold, 
\[
\begin{aligned}&\left\{ (\zeta_{f} (\bm{x} ),\eta^{\alpha }_{g} (\bm{x} ))\in [\sigma_{1} ,1]^{2}\; \middle| \; \eta^{\alpha }_{g} (\bm{x} )\geq \gamma (\zeta_{f} (\bm{x} ))\right\}  \\ &\subset \left\{ (\zeta_{f} (\bm{x} ),\eta^{\alpha }_{g} (\bm{x} ))\in [\sigma_{1} ,1]^{2}\; \middle| \; \zeta_{f} (\bm{x} )\eta^{\alpha }_{g} (\bm{x} )\geq \sigma_{1} \right\}  \subset [\sigma_{1} ,1]^{2}.\end{aligned} 
\]

\textit{Lemma \ref{lem1}} characterizes a fundamental constraint on the admissible pairs 
$(\zeta_f(\bm{x}),\eta_g^{\alpha}(\bm{x}))$ in the upper-right region of the unit square. 
By symmetry, analogous results can be derived for the remaining three corners.

For a subset $\mathcal{S}\subseteq\mathcal{V}$, denote its complement by 
$\mathcal{S}^c$, such that $\mathcal{V}=\mathcal{S}\cup\mathcal{S}^c$ and 
$\mathcal{S}\cap\mathcal{S}^c=\emptyset$. The corresponding complementary 
vertex filter is defined as $\bar{f}_i=1-f_i$. Similarly, for 
$\mathcal{F}\subseteq\hat{\mathcal{G}}$, define $\bar{g}_j=1-g_j$.

To distinguish the four corner cases, we explicitly denote by 
$\sigma_1(\mathbf{S}^{\alpha}_{f,g})$ the largest eigenvalue of the joint 
operator $\mathbf{S}^{\alpha}_{f,g}$. By considering all subregions of 
$[0,1]^2$, we obtain the following GFRFT-based uncertainty principle.

\begin{thm}\label{thm3}
There exists a signal $\bm{x}$ satisfying $\|\bm{x}\|_2=1$,
	$\|\mathbf{D}_f \bm{x}\|_2^2=\zeta_f(\bm{x})$, and 
	$\|\mathbf{B}_g^{\alpha}\bm{x}\|_2^2=\eta_g^{\alpha}(\bm{x})$, such that
	$(\zeta_f(\bm{x}),\eta_g^{\alpha}(\bm{x})) \in \Gamma(\mathbf{D}_f,\mathbf{B}_g^{\alpha})$, 
	where
\begin{equation}
	\Gamma(\mathbf{D}_f,\mathbf{B}_g^{\alpha})
	:=
	\left\{
	(\zeta_f(\bm{x}),\eta_g^{\alpha}(\bm{x})) \in [0,1]^2
	\;\middle|\;
	\begin{aligned}
		\eta_g^{\alpha}(\bm{x}) \le \gamma(\zeta_f(\bm{x})), 
		&\quad \text{if } \zeta_f(\bm{x}) \eta_g^{\alpha}(\bm{x}) \ge \sigma_1(\mathbf{S}^{\alpha}_{f,g}),\\
		1-\eta_{\bar{g}}^{\alpha}(\bm{x}) \le \gamma(\zeta_f(\bm{x})), 
		&\quad \text{if } \zeta_f(\bm{x}) (1-\eta_{\bar{g}}^{\alpha}(\bm{x})) \ge \sigma_1(\mathbf{S}^{\alpha}_{f,\bar{g}}),\\
		\eta_g^{\alpha}(\bm{x}) \le \gamma(1-\zeta_{\bar{f}}(\bm{x})), 
		&\quad \text{if } (1-\zeta_{\bar{f}}(\bm{x}))\eta_g^{\alpha}(\bm{x}) \ge \sigma_1(\mathbf{S}^{\alpha}_{\bar{f},g}),\\
		1-\eta_{\bar{g}}^{\alpha}(\bm{x}) \le \gamma(1-\zeta_{\bar{f}}(\bm{x})), 
		&\quad \text{if } (1-\zeta_{\bar{f}}(\bm{x}))(1-\eta_{\bar{g}}^{\alpha}(\bm{x})) \ge \sigma_1(\mathbf{S}^{\alpha}_{\bar{f},\bar{g}})
	\end{aligned}
	\right\}.
\end{equation}
\end{thm}

\begin{proof}
	The proof follows directly from \textit{Lemma \ref{lem1}} by symmetry, applying the same argument to the complementary filters and the corresponding corner regions.
\end{proof}

For the four corner regions, when $\sigma_1 < 1$, \textit{Theorem \ref{thm3}} establishes an uncertainty relation between the operators $\mathbf{D}_f$ and $\mathbf{B}_g^{\alpha}$. It implies that a graph signal $\bm{x}$ cannot be simultaneously well localized with respect to both operators. In particular, the admissible pairs $(\zeta_f(\bm{x}), \eta_g^{\alpha}(\bm{x}))$ cannot approach the corner $(1,1)$. In the special case of \textit{Example \ref{example4}}, where $\mathbf{D}_f$ and $\mathbf{B}_g^{\alpha}$ are defined via energy concentration, the uncertainty characterization becomes sharp. The associated operators quantify localization in a set theoretic manner. Specifically, if $\bm{x}$ is supported on a vertex subset $\mathcal{S} \subseteq \mathcal{V}$, then $\zeta_f(\bm{x}) = 1$, whereas $\zeta_f(\bm{x}) = 0$ if $\bm{x}$ is supported on $\mathcal{S}^c$. Since the class of signals supported on $\mathcal{S}$ is large, this setting admits a broader family of signals that exhibit joint localization in both vertex and graph fractional spectral domains. Consequently, the feasible region $\Gamma(\mathbf{D}_f,\mathbf{B}_g^{\alpha})$ is generally larger than that of other operator pairs.

\begin{remark}
	In contrast to classical time-frequency analysis, where perfect joint localization is impossible, the graph fractional setting may admit the case $\sigma_1 = 1$. In this scenario, there exists a signal $\bm{x}$ such that $\mathbf{D}_f \bm{x} = \mathbf{B}_g^{\alpha} \bm{x} = \bm{x}$, or equivalently $\zeta_f(\bm{x}) = \eta_g^{\alpha}(\bm{x}) = 1$, indicating perfect localization in both domains. Therefore, a nontrivial graph fractional uncertainty principle holds only when $\sigma_1 < 1$, which excludes such degenerate cases.
\end{remark}

Since $\mathbf{S}^{\alpha}_{f,g}$ is Hermitian, its eigenvectors 
$\{\bm{\psi}_n\}_{n=1}^N$ form an orthonormal basis with corresponding eigenvalues $\{\sigma_n\}_{n=1}^N$. Any signal $\bm{x} \in \mathbb{C}^N$ can therefore be expanded as $\bm{x} = \sum_{n=1}^{N} (\bm{\psi}_n^{\mathrm{H}}\bm{x}) \bm{\psi}_n$. Let $\varkappa < \sigma_1$, then, the approximation of $\bm{x}$ obtained by retaining only the components corresponding to eigenvalues of set $\mathcal{I}_{\varkappa}
:= \{ n \in \{1,\dots,N\} \mid \sigma_n \ge \varkappa \}$ obeys the error bound
\begin{equation}
	\Big\|
	\bm{x} - \sum_{n \in \mathcal{I}_{\varkappa}}
	(\bm{\psi}_n^{\mathrm{H}}\bm{x}) \bm{\psi}_n
	\Big\|_2^2
	\le
	\frac{\sigma_1- \varkappa^{\alpha}_{f,g}(\bm{x})}
	{\sigma_1 - \varkappa}
	\|\bm{x}\|_2^2,
\end{equation}
where $\varkappa^{\alpha}_{f,g}(\bm{x})$ denotes the mean localization level defined in Eq.~\eqref{varkappa}. Furthermore, for any $a > 0$, define the set $\mathcal{I}_{\varkappa,a}
:=
\{
n \in \{1,\dots,N\}
\;|\;
\sigma_n \in [\varkappa^{\alpha}_{f,g}(\bm{x}) - a,\;
\varkappa^{\alpha}_{f,g}(\bm{x}) + a]
\}$.
Then, the approximation satisfies
\begin{equation}
	\Big\|
	\bm{x} - \sum_{n \in \mathcal{I}_{\varkappa,a}}
	(\bm{\psi}_n^{\mathrm{H}}\bm{x}) \bm{\psi}_n
	\Big\|_2^2
	\le
	\frac{\mathrm{var}[\mathbf{S}^{\alpha}_{f,g}](\bm{x})}{a^2}
	\|\bm{x}\|_2^2,
\end{equation}
where the variance is defined as
\[
\mathrm{var}[\mathbf{S}^{\alpha}_{f,g}](\bm{x})
=
\bm{x}^{\mathrm{H}}
\left(\mathbf{S}^{\alpha}_{f,g}
-
\varkappa^{\alpha}_{f,g}(\bm{x})\mathbf{I}_N\right)^2
\bm{x}.
\]

This observation establishes a direct connection between the uncertainty region $\Gamma(\mathbf{D}_f,\mathbf{B}_g^{\alpha})$ and signal representation, as signals associated with boundary points of $\Gamma(\mathbf{D}_f,\mathbf{B}_g^{\alpha})$ exhibit strong joint localization and thus admit efficient approximations using only a few eigenvectors of $\mathbf{S}^{\alpha}_{f,g}$.

\subsection{Uncertainty Principles via Rotated Operators}

To further characterize the admissible region $\Gamma(\mathbf{D}_f,\mathbf{B}_g^{\alpha})$, we adopt a geometric perspective based on the numerical range. This viewpoint provides a unified framework to describe the trade-off between vertex domain and graph fractional spectral domain localization.

For a normalized graph signal $\bm{x} \in \mathbb{C}^N$ with $\|\bm{x}\|_2 = 1$, recall the localization measures in Eq.~\eqref{zeta} and \eqref{eta}. The admissible pair $(\zeta_f(\bm{x}), \eta_g^{\alpha}(\bm{x}))$ can be identified with a point in $\mathbb{R}^2$, or equivalently with the complex scalar
\[
\zeta_f(\bm{x}) + \mathrm{i}\,\eta_g^{\alpha}(\bm{x})
=
\bm{x}^{\mathrm{H}} \left( \mathbf{D}_f + \mathrm{i}\mathbf{B}_g^{\alpha} \right) \bm{x},
\]
where $\mathrm{i}$ denotes the imaginary unit.

This observation establishes a direct connection between the admissible region $\Gamma(\mathbf{D}_f,\mathbf{B}_g^{\alpha})$ and the numerical range of the operator
\[
\mathbf{A}^{\alpha}_{f,g} := \mathbf{D}_f + \mathrm{i}\mathbf{B}_g^{\alpha}.
\]

\begin{defn}
	The numerical range of $\mathbf{A}^{\alpha}_{f,g}$ is defined as
	\[
	\mathcal{C}(\mathbf{A}^{\alpha}_{f,g})
	:=
	\left\{
	\bm{x}^{\mathrm{H}} \mathbf{A}^{\alpha}_{f,g} \bm{x}
	\;\middle|\;
	\bm{x} \in \mathbb{C}^N,\ \|\bm{x}\|_2 = 1
	\right\}.
	\]
\end{defn}
By identifying $\mathbb{C}$ with $\mathbb{R}^2$, the admissible region $\Gamma(\mathbf{D}_f,\mathbf{B}_g^{\alpha})$ corresponds to the real-imaginary projection of $\mathcal{C}(\mathbf{A}^{\alpha}_{f,g})$. A fundamental property of the numerical range is convexity, as guaranteed by the Hausdorff-Toeplitz theorem \cite{Hausdorff}. Consequently, $\Gamma(\mathbf{D}_f,\mathbf{B}_g^{\alpha})$ is a convex and compact subset of $[0,1]^2$.

To characterize its boundary, we introduce a family of rotated joint operators, defined in Eq.~\eqref{R}, and since $\mathbf{R}^{\alpha}_{f,g}(\beta)$ is Hermitian, it admits a complete set of orthonormal eigenvectors. Let $\rho_1(\beta)$ denote its largest eigenvalue and $\bm{\phi}_1(\beta)$ a corresponding eigenvector. Here, we define the supporting line $\ell(\beta) :=
\big\{
(\zeta_f(\bm{x}),\eta_g^{\alpha}(\bm{x})) \in \mathbb{R}^2 \;\big|\;
\cos\beta\,\zeta_f(\bm{x}) + \sin\beta\,\eta_g^{\alpha}(\bm{x}) = \rho_1(\beta)\big\}$. The following result provides a geometric uncertainty principle in the GFRFT domain.

\begin{thm}\label{thm4}
	For any $\beta \in [0,2\pi)$, the admissible region satisfies
	\begin{equation}
		\begin{aligned}\Gamma (\mathbf{D}_{f},\mathbf{B}^{\alpha }_{g} )\subseteq &\big\{ (\zeta_{f}(\bm{x}),\eta^{\alpha }_{g}(\bm{x}) )\in [0,1]^{2}\\ & \;\;\big|\; \cos \beta \, \zeta_{f}(\bm{x}) +\sin \beta \, \eta^{\alpha }_{g}(\bm{x}) \leq \rho_{1} (\beta )\big\}, \end{aligned} 
	\end{equation}
	where the boundary is supported by the line $\ell(\beta)$.
\end{thm}
\begin{proof}
	The proof is provided in Appendix~\ref{appendixD}.
\end{proof}
The inequality becomes tight when $\bm{x}$ is chosen as an eigenvector corresponding to $\rho_1(\beta)$. Moreover, for every boundary point $(\zeta_f(\bm{x}),\eta_g^{\alpha}(\bm{x}))$ of $\Gamma(\mathbf{D}_f,\mathbf{B}_g^{\alpha})$, there exists an angle $\beta$ such that
\[
\cos\beta\, \zeta_f(\bm{x}) + \sin\beta\, \eta_g^{\alpha}(\bm{x})= \rho_1(\beta),
\]
and this point is attained by an eigenvector of $\mathbf{R}^{\alpha}_{f,g}$, yielding
\[
(\zeta_f(\bm{x}),\eta_g^{\alpha}(\bm{x}))
=
\left(
\bm{\phi}^{\mathrm{H}}_1(\beta) \mathbf{D}_f \bm{\phi}_1(\beta),
\;
\bm{\phi}^{\mathrm{H}}_1(\beta) \mathbf{B}_g^{\alpha} \bm{\phi}_1(\beta)
\right).
\]

The operator $\mathbf{R}^{\alpha}_{f,g}$ induces a family of linear trade-offs between vertex and graph fractional spectral localization. The maximal eigenvalue $\rho_1(\beta)$ defines the supporting line $\ell(\beta)$ which has slope $-\cot\beta$ in the $(\zeta_f(\bm{x}),\eta_g^{\alpha}(\bm{x}))$ plane. The admissible region is obtained as the intersection of all such supporting half-spaces.

\begin{remark}
	When $\beta = 0$ and $\beta = \frac{\pi}{2}$, the operator reduces to $\mathbf{D}_f$ and $\mathbf{B}_g^{\alpha}$, respectively. In the balanced case $\beta = \frac{\pi}{4}$, one obtains the symmetric bound
	\[
	\zeta_f(\bm{x}) + \eta_g^{\alpha}(\bm{x})
	\le \sqrt{2}\,\rho_1\!\left(\tfrac{\pi}{4}\right),
	\]
	where $\rho_1\!\left(\tfrac{\pi}{4}\right)$ is the largest eigenvalue of
	$\frac{1}{\sqrt{2}}(\mathbf{D}_f + \mathbf{B}_g^{\alpha})$. This yields a balanced uncertainty relation in the GFRFT domain.
\end{remark}

For any fixed $\beta \in [0,2\pi)$, the rotated operator $\mathbf{R}^{\alpha}_{f,g}$ is Hermitian and admits an orthonormal eigenbasis $\{\bm{\phi}_n(\beta)\}_{n=1}^N$ with eigenvalues $\{\rho_n(\beta)\}_{n=1}^N$. Accordingly, any signal $\bm{x}$ can be expanded as $\bm{x} = \sum_{n=1}^{N} (\bm{\phi}_n^{\mathrm{H}}(\beta)\bm{x}) \bm{\phi}_n(\beta)$. Let $\kappa < \rho_1(\beta)$. Then, truncating the expansion to the components corresponding to set $\mathcal{I}_{\kappa}
:= \{ n \in \{1,\dots,N\} \mid \rho_n(\beta) \ge \kappa \}$ yields the approximation error bound
\begin{equation}
	\Big\|
	\bm{x} - \sum_{n\in \mathcal{I}_{\kappa}}
	(\bm{\phi}_n^{\mathrm{H}}(\beta)\bm{x}) \bm{\phi}_n(\beta)
	\Big\|_2^2
	\le
	\frac{\rho_1(\beta) - \kappa^{\alpha}_{f,g}(\bm{x})}
	{\rho_1(\beta) - \kappa}
	\|\bm{x}\|_2^2,
\end{equation}
where $\kappa^{\alpha}_{f,g}(\bm{x})$ is defined in Eq.~\eqref{kappa}. Moreover, for any $a > 0$, define the set $\mathcal{I}_{\kappa,a}
:=
\{ n \in \{1,\dots,N\}\;|\; \rho_n(\beta) \in [\kappa^{\alpha}_{f,g}(\bm{x}) - a,\;
\kappa^{\alpha}_{f,g}(\bm{x}) + a]
\}$. Then, restricting the expansion to eigenvalues $\rho_n(\beta)$ gives
\begin{equation}
	\Big\|
	\bm{x} - \sum_{n \in \mathcal{I}_{\kappa,a}}
	(\bm{\phi}_n^{\mathrm{H}}(\beta)\bm{x}) \bm{\phi}_n(\beta)
	\Big\|_2^2
	\le 	
	\frac{\mathrm{var}[\mathbf{R}^{\alpha}_{f,g}](\bm{x})}{a^2}	\|\bm{x}\|_2^2,
\end{equation}
where
\[
\mathrm{var}[\mathbf{R}^{\alpha}_{f,g}](\bm{x})
=
\bm{x}^{\mathrm{H}}
\left(\mathbf{R}^{\alpha}_{f,g}
-
\kappa^{\alpha}_{f,g}(\bm{x})\mathbf{I}_N\right)^2
\bm{x}.
\]

The dominant eigenvector $\bm{\phi}_1(\beta)$ achieves the maximal projection onto $\ell(\beta)$ and represents the most localized signal along direction $\beta$. 
If $\bm{x}$ aligns with a given trade-off direction, its energy concentrates on a few leading eigenvectors of $\mathbf{R}^{\alpha}_{f,g}$. 
Thus, by selecting $\beta$, the approximation adapts to the geometry of the uncertainty region, with the error governed by spectral concentration. 
The rotated operator family therefore not only characterizes the boundary of $\Gamma(\mathbf{D}_f,\mathbf{B}_g^{\alpha})$, but also enables geometry-adaptive representations.

%主导特征向量 $\bm{\phi}_1(\beta)$ 对应于在支撑线 $\ell(\beta)$ 上的最大投影，因此代表了沿 $\beta$ 指定方向的最局部化信号。因此，如果信号 $\bm{x}$ 与某个特定的权衡方向对齐，其能量将集中在 $\mathbf{R}^{\alpha}_{f,g}(\beta)$ 的几个主要特征向量上。因此，通过选择合适的旋转参数 $\beta$，可以使 $\bm{x}$ 的近似值适应不确定区域的几何形状。在这种情况下，近似误差取决于信号在 $\mathbf{R}^{\alpha}_{f,g}(\beta)$ 的谱中的集中程度，更强的方向局部化会导致更有效的表示。这种观点强调，旋转算子族不仅刻画了 $\Gamma(\mathbf{D}_f,\mathbf{B}_g^{\alpha})$ 的边界，而且还提供了一种机制来构建适应不同定位权衡的信号表示。

%The leading eigenvector $\bm{\phi}_1(\beta)$ corresponds to the maximal projection onto the supporting line $\ell(\beta)$ and therefore represents the most localized signal along the direction specified by $\beta$. 
%As a result, if a signal $\bm{x}$ is aligned with a particular trade-off direction, its energy becomes concentrated in a few dominant eigenvectors of $\mathbf{R}^{\alpha}_{f,g}(\beta)$.
%
%Consequently, the approximation of $\bm{x}$ can be adapted to the geometry of the uncertainty region by selecting an appropriate rotation parameter $\beta$. 
%In this case, the approximation error is governed by how concentrated the signal is in the spectrum of $\mathbf{R}^{\alpha}_{f,g}(\beta)$, with stronger directional localization leading to more efficient representations.
%
%This perspective highlights that the family of rotated operators not only characterizes the boundary of $\Gamma(\mathbf{D}_f,\mathbf{B}_g^{\alpha})$, but also provides a mechanism to construct signal representations adapted to different localization trade-offs.

\subsection{Numerical Approximation of the Admissible Region}

The geometric characterization in \textit{Theorem \ref{thm4}} provides a complete description of the admissible region $\Gamma(\mathbf{D}_f,\mathbf{B}_g^{\alpha})$ as the intersection of supporting half-spaces induced by the family of operators $\mathbf{R}^{\alpha}_{f,g}$. In practice, however, an explicit representation of this region is not directly available, and numerical approximation becomes necessary.

To this end, we consider a finite set of angles 
$\boldsymbol{\beta}=\{\beta_1,\dots,\beta_M\} \subset [0,2\pi)$ with $M \ge 3$. 
For each $\beta_m$, let $\rho_1(\beta_m)$ denote the largest eigenvalue of $\mathbf{R}^{\alpha}_{f,g}(\beta_m)$, and let $\bm{\phi}_1(\beta_m)$ be a corresponding eigenvector. We define the boundary point
\[
p(\beta_m)
:=
\Big(
\bm{\phi}_1^{\mathrm{H}}(\beta_m)\mathbf{D}_f\bm{\phi}_1(\beta_m),
\;
\bm{\phi}_1^{\mathrm{H}}(\beta_m)\mathbf{B}_g^{\alpha}\bm{\phi}_1(\beta_m)
\Big),
\]
which lies on the supporting line $\ell(\beta_m)$. Using these quantities, we construct two approximations of the admissible region.

%First, an outer approximation is obtained as the intersection of the half-spaces induced by the supporting lines,
%\[
%\begin{aligned}\Gamma_{\mathrm{out}}(\bm{\beta}) :=\bigcap^{M}_{m=1} 
%&\big\{ (\zeta_{f}(\bm{x}) ,\eta^{\alpha }_{g}(\bm{x}) )\in [0,1]^{2}\\ 
%&\;\;\big|\; \cos \beta_{m} \, \zeta_{f}(\bm{x}) +\sin \beta_{m} \, \eta^{\alpha }_{g}(\bm{x}) \leq \rho_{1} (\beta_{m} )\big\}. \end{aligned}
%\]
%
%Second, an inner approximation is constructed as the convex hull of the boundary points,
%\[
%\Gamma_{\mathrm{in}}(\bm{\beta})
%:=
%\operatorname{conv}
%\big\{
%p(\beta_1),\dots,p(\beta_M)
%\big\}.
%\]
%
%Due to the convexity of $\Gamma(\mathbf{D}_f,\mathbf{B}_g^{\alpha})$, it follows directly that
%\[
%\Gamma_{\mathrm{in}}(\bm{\beta})
%\subseteq
%\Gamma(\mathbf{D}_f,\mathbf{B}_g^{\alpha})
%\subseteq
%\Gamma_{\mathrm{out}}(\bm{\beta}).
%\]
First, an outer approximation is obtained as the intersection of the half-spaces induced by the supporting lines,
\[
\begin{aligned}
	\Gamma_{\mathrm{out}}(\bm{\beta}) 
	:=\bigcap^{M}_{m=1} 
	&\big\{ (\zeta_{f}(\bm{x}) ,\eta^{\alpha }_{g}(\bm{x}) )\in [0,1]^{2}\\ 
	&\;\;\big|\; \cos \beta_{m} \, \zeta_{f}(\bm{x}) +\sin \beta_{m} \, \eta^{\alpha }_{g}(\bm{x}) \leq \rho_{1} (\beta_{m} )\big\}. 
\end{aligned}
\]

Second, an inner approximation is constructed as the convex hull of the boundary points, obtained as the smallest convex set containing all points,
\[
\Gamma_{\mathrm{in}}(\bm{\beta})
:=
\operatorname{conv}
\big\{
p(\beta_1),\dots,p(\beta_M)
\big\}.
\]

Due to the convexity of $\Gamma(\mathbf{D}_f,\mathbf{B}_g^{\alpha})$, it follows directly that
\[
\Gamma_{\mathrm{in}}(\bm{\beta})
\subseteq
\Gamma(\mathbf{D}_f,\mathbf{B}_g^{\alpha})
\subseteq
\Gamma_{\mathrm{out}}(\bm{\beta}).
\]
This establishes a sandwich-type approximation of the admissible region. As the number of angles $M$ increases and the sampling of $\beta$ becomes denser, the inner and outer approximations converge to the true admissible region.

Finally, we note that this polygonal approximation framework \cite{GshapesUC} naturally extends the uncertainty principle in the GFRFT domain from a theoretical characterization to a numerically tractable form, enabling visualization and quantitative analysis of the trade-off between vertex and graph fractional spectral localization. We summarize this numerical approximately in Algorithm \ref{alg1}.

\begin{algorithm}[h!]
	\small
	\caption{Numerical approximation of $\Gamma(\mathbf{D}_f,\mathbf{B}_g^{\alpha})$}
	\label{alg1}
	\begin{algorithmic}
		
		\State \textbf{Input:} $\mathbf{D}_f$, $\mathbf{B}_g^{\alpha}$, angles $\{\beta_m\}_{m=1}^M \subset [0,2\pi)$
		
		\For{$m = 1,\dots,M$}
		\State Construct
		\[
		\mathbf{R}^{\alpha}_{f,g}(\beta_m)
		=
		\cos\beta_m\,\mathbf{D}_f
		+
		\sin\beta_m\,\mathbf{B}_g^{\alpha}
		\]
		\State Compute largest eigenvalue $\rho_1(\beta_m)$ and eigenvector $\bm{\phi}_1(\beta_m)$
		\State Compute boundary point
		\[
		p(\beta_m)
		=
		\Big(
		\bm{\phi}_1^{\mathrm{H}}(\beta_m)\mathbf{D}_f\bm{\phi}_1(\beta_m),
		\;
		\bm{\phi}_1^{\mathrm{H}}(\beta_m)\mathbf{B}_g^{\alpha}\bm{\phi}_1(\beta_m)
		\Big)
		\]
		\EndFor
		
		\State Construct inner approximation
		\[
		\Gamma_{\mathrm{in}}(\boldsymbol{\beta})
		=
		\operatorname{conv}
		\{p(\beta_1), \dots, p(\beta_M)\}
		\]
		
		\For{$m = 1,\dots,M$}
		\State Compute $q(\beta_m)$ as the intersection of $\ell(\beta_{m-1})$ and $\ell(\beta_m)$
		\EndFor
		
		\State Construct outer approximation
		\[
		\Gamma_{\mathrm{out}}(\boldsymbol{\beta})
		=
		\operatorname{conv}
		\{q(\beta_1), \dots, q(\beta_M)\}
		\]
		
		\State \textbf{Output:} $\Gamma_{\mathrm{in}}(\boldsymbol{\beta})$, $\Gamma_{\mathrm{out}}(\boldsymbol{\beta})$
	\end{algorithmic}
\end{algorithm}

\section{Examples and Illustrations}
\label{sec5}

This section provides numerical examples to illustrate the proposed graph fractional uncertainty principle. We examine how the fractional order $\alpha$ changes graph basis functions and the corresponding admissible uncertainty regions. Then, we study how different localization filters $\bm{f}$ and $\bm{g}$ affect the uncertainty geometry on a fixed graph. Furthermore, we analyze the spectra of the joint localization operators to give a quantitative explanation of the observed uncertainty regions.

\subsection{Effect of the Fractional Order on Graph Uncertainty Regions}

We first investigate the role of the fractional order $\alpha$. Two representative graphs are considered \cite{GSPBox}. The first graph is an Erd\H{o}s--R\'enyi graph with $N=256$ vertices and connection threshold $1/6$, and it is used under \textit{Example~\ref{example1}}. The second graph is a guppy graph\footnote{[Online]. Available: \url{https://github.com/WolfgangErb/GUPPY}.} with $N=896$ vertices, and its geometry is more irregular, used for \textit{Example~\ref{example2}}. For graph signal visualization, the fractional order is fixed as $\alpha=0.8$. The displayed GFT-domain signal is generated as a linear combination of the first three GFT basis vectors with weights $1$, $0.5$, and $2$, respectively. Similarly, the displayed GFRFT-domain signal is generated from the first three GFRFT basis vectors using the same weights, and its real part is used for visualization.

Fig. \ref{fig01} compares representative GFT and GFRFT basis signals. Panels (a) and (b) show the GFT-domain and GFRFT-domain basis signals on the Erd\H{o}s--R\'enyi graph, respectively. Panels (c) and (d) show the same comparison on the guppy graph. The GFRFT basis changes the spatial distribution of graph spectral atoms, indicating that the fractional order provides an additional degree of freedom for adjusting the localization pattern of graph signals.

\begin{figure}[h!]
	\centering
	\includegraphics[width=0.55\linewidth]{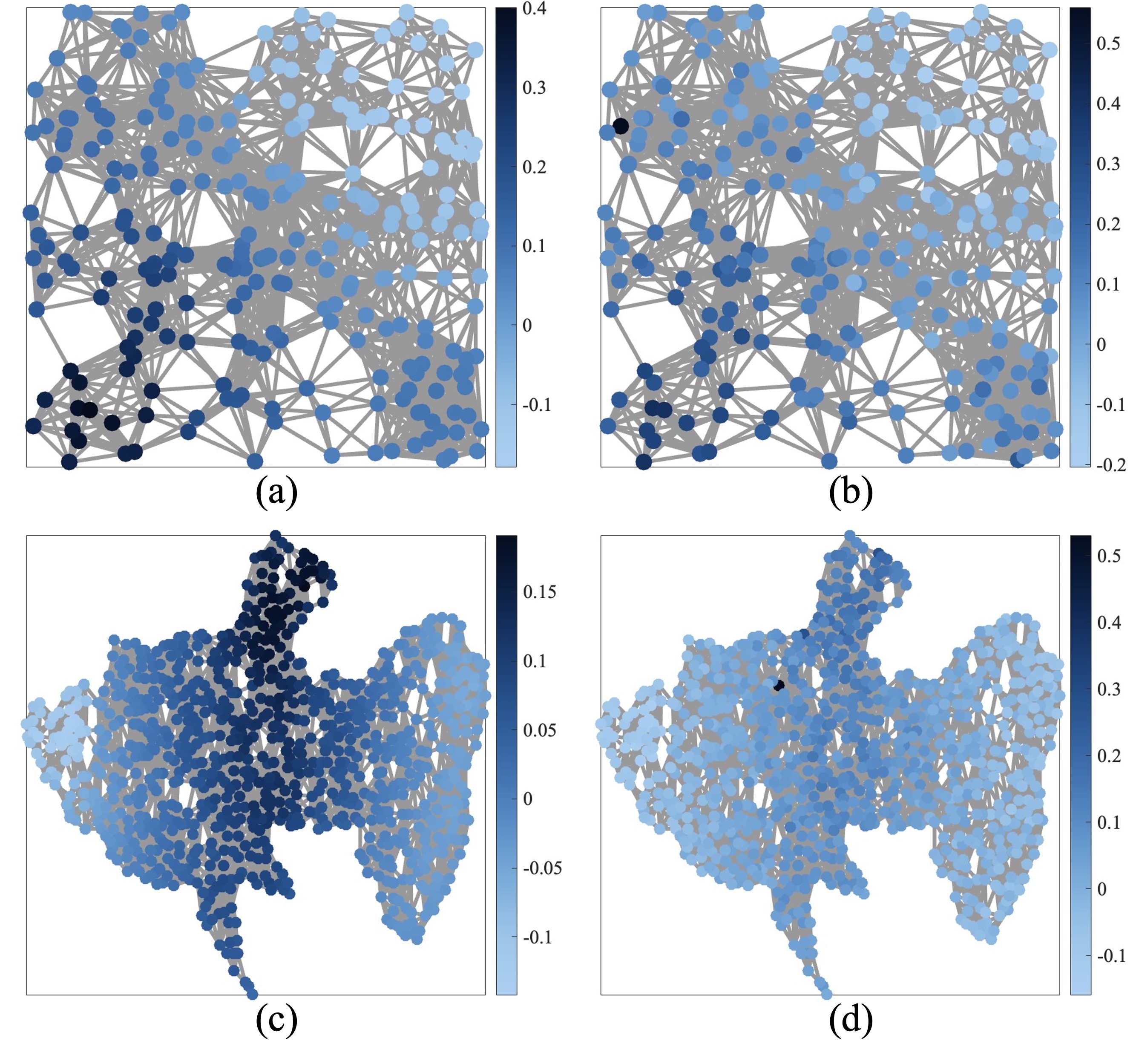}
	\vspace*{-8pt}
	\caption{Representative GFT and GFRFT signals: (a) Erdős–Rényi GFT, (b) Erdős–Rényi GFRFT, (c) guppy GFT, (d) guppy GFRFT.}
	\label{fig01}
\end{figure}

We next compute the admissible uncertainty regions for different fractional orders. For each order $\alpha$, the vertex-domain and fractional spectral localization operators are constructed from the prescribed filters. The boundary of each uncertainty region is obtained using a polygonal numerical-range approximation with 240 sampling directions and a rotation parameter of $\pi/4$. For the Erd\H{o}s--R\'enyi graph under \textit{Example~\ref{example1}}, the reference vertex is chosen as the middle-index vertex. The vertex-domain filter adopts a distance-decay exponent $a=1$, whereas the fractional spectral filter uses a soft-bandlimited smoothing exponent $b=2$. The fractional spectral subset $\mathcal{F}$ consists of the first 47 low-frequency components. For the guppy graph under \textit{Example~\ref{example2}}, both localization filters are generated from the corresponding fractional Laplacian eigenvalue structure, without introducing an additional projection subset. The case $\alpha=1$ recovers the conventional GFT-based setting and is included as a reference.

Figs.~\ref{er} and \ref{guppy} show the uncertainty regions on the Erd\H{o}s--R\'enyi graph under \textit{Example~\ref{example1}} and on the guppy graph under \textit{Example~\ref{example2}}, respectively. In both figures, the first seven panels present the uncertainty regions for individual fractional orders, while the last panel overlays all boundaries. For $\alpha=1$, the corresponding GFT-based uncertainty region is also plotted as a reference. The close agreement between the GFRFT-based and GFT-based regions at $\alpha=1$ confirms that the proposed formulation recovers the classical GFT-domain case. Compared with this GFT-based reference, varying the fractional order reduces the admissible uncertainty region in Fig.~\ref{er}, whereas it enlarges the region in Fig.~\ref{guppy}. This graph-dependent behavior indicates that the fractional order can reshape the uncertainty region beyond the classical case, providing a tunable mechanism for controlling the vertex-spectral localization trade-off.
\begin{figure*}[h!]
	\centering
	\includegraphics[width=\linewidth]{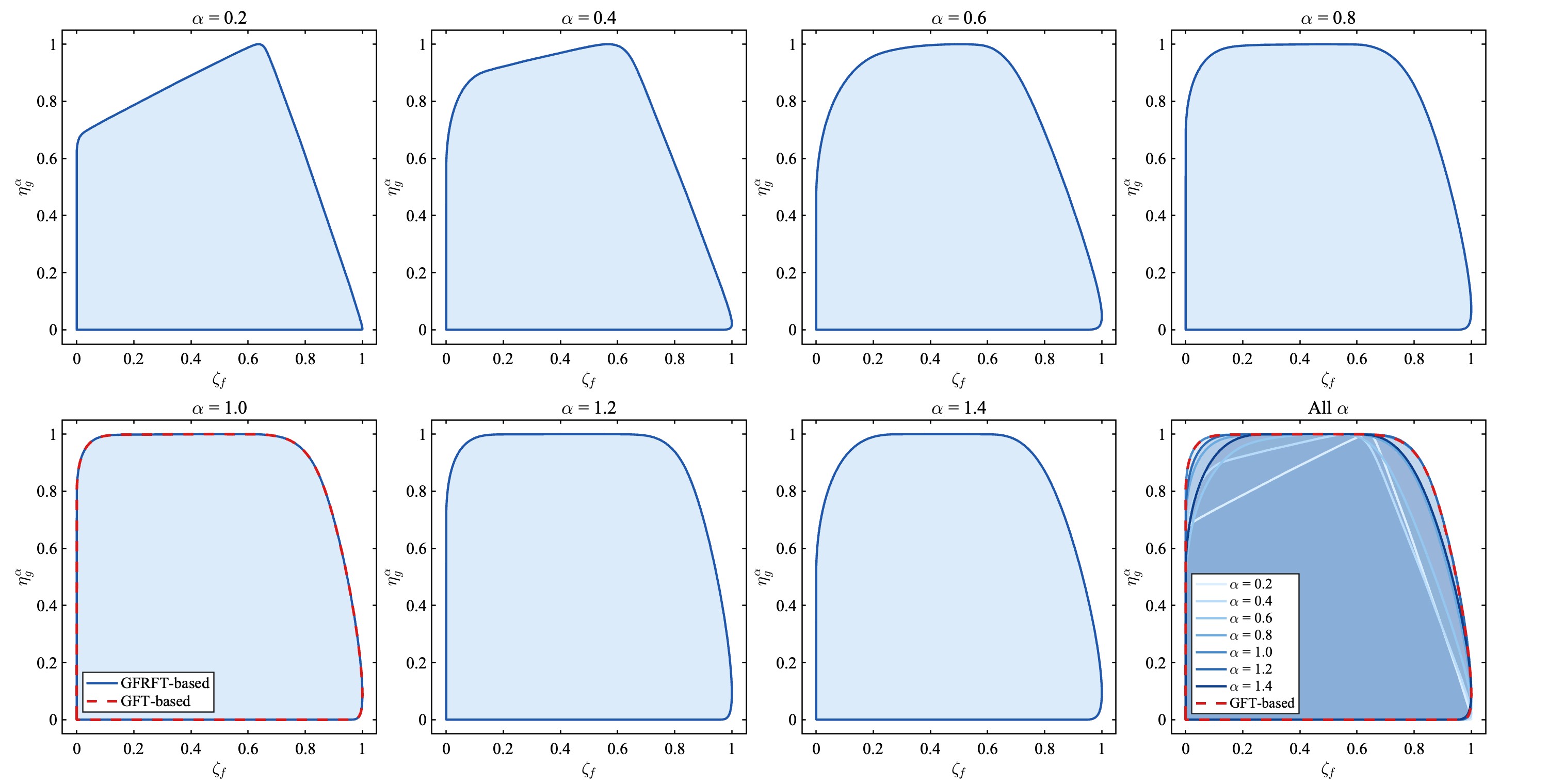}
	\vspace*{-15pt}
	\caption{Effect of the fractional order $\alpha$ on the uncertainty regions for the Erd\H{o}s--R\'enyi graph under \textit{Example~\ref{example1}}. The first seven panels correspond to $\alpha=0.2,0.4,0.6,0.8,1.0,1.2,1.4$, respectively, while the last panel overlays all GFRFT-based boundaries.}
	\label{er}
\end{figure*}
\begin{figure*}[h!]
	\centering
	\includegraphics[width=\linewidth]{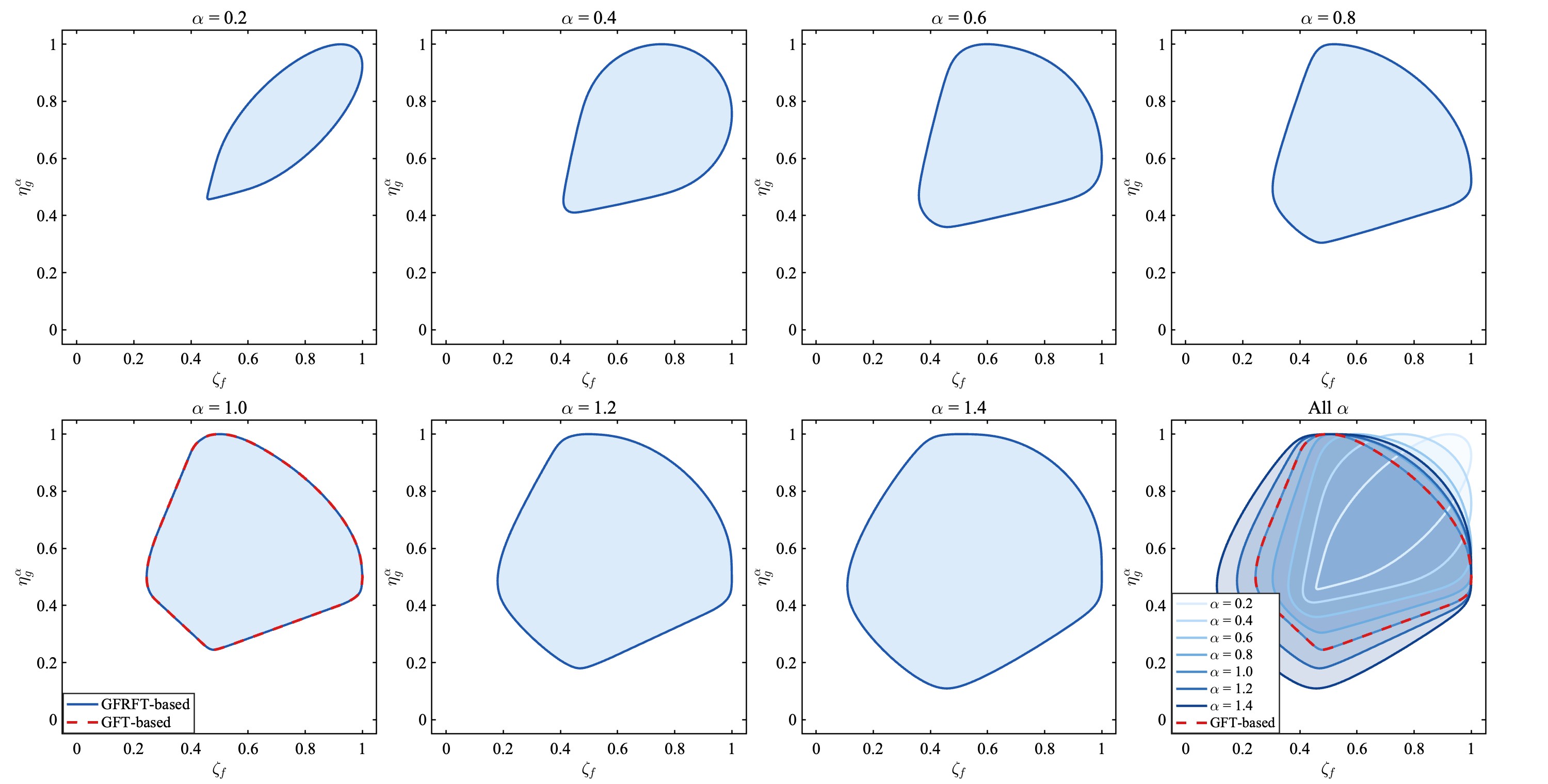}
	\vspace*{-15pt}
	\caption{Effect of the fractional order $\alpha$ on the uncertainty regions for the guppy graph under \textit{Example~\ref{example2}}. The first seven panels correspond to $\alpha=0.2,0.4,0.6,0.8,1.0,1.2,1.4$, respectively, while the last panel overlays all GFRFT-based boundaries.}
	\label{guppy}
\end{figure*}

\subsection{Influence of Localization Filters in the GFRFT Domain}

We then investigate the effect of graph fractional localization filters on the bunny graph\footnote{[Online]. Available: \url{https://graphics.stanford.edu/data/3Dscanrep/}.} with $N=900$, and the fractional order is fixed at $\alpha=0.7$. Four representative filter pairs, denoted by \textit{Examples~\ref{example1}--\ref{example4}}, are considered. Each example defines a vertex-domain filter $\bm{f}$ and a fractional spectral filter $\bm{g}$. The vertex filter is centered at the geometric center of the graph. For vertex localization based on distance or projection, the radius is set to $0.015$ to obtain a nondegenerate local vertex set. The distance-decay exponent for vertex localization is set to 1, and the soft-bandlimited smoothing exponent for fractional spectral localization is set to 2. When a bandlimited spectral construction is used, the fractional spectral subset includes the first $162$ low-frequency components.

Fig. \ref{filterbunny} illustrates the vertex-domain filters $\bm{f}$ for the four examples. The panels reveal how different examples define distinct notions of vertex localization: some filters are tightly concentrated around the center, whereas others exhibit smoother or broader support. Fig. \ref{filterhatg} shows the corresponding fractional spectral representations $|\hat{\bm{g}}|=|\mathbf{F}^{\alpha} \bm{g}|$. Together, these figures clarify the source of differing uncertainty regions: each example combines a specific vertex-domain and fractional spectral filter.
\begin{figure*}[h!]
	\centering
	\includegraphics[width=\linewidth]{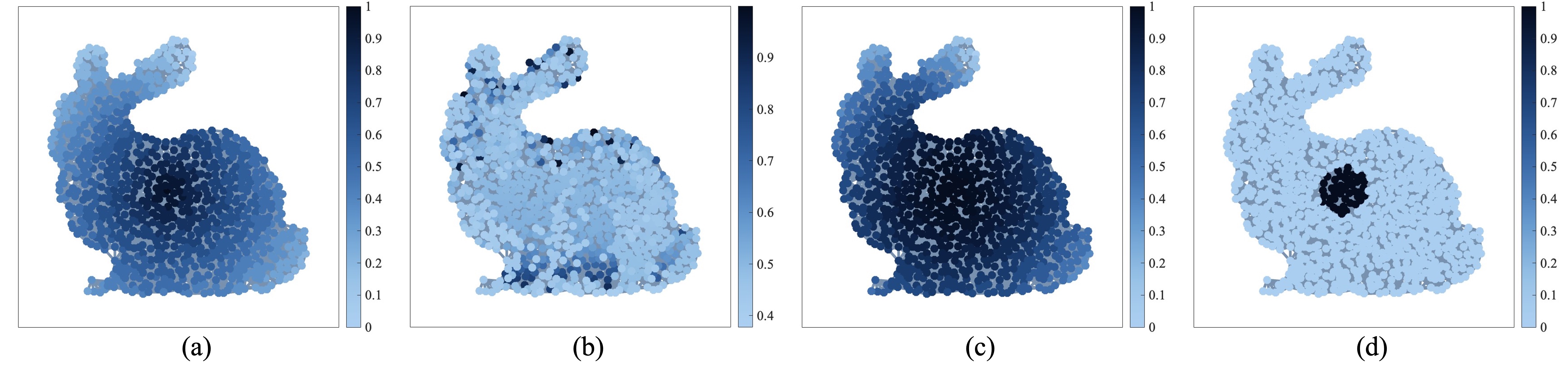}
	\vspace*{-20pt}
	\caption{Vertex-domain filter signals $\bm{f}$ for four localization examples on the bunny graph. Panels (a)--(d) correspond to \textit{Examples \ref{example1}--\ref{example4}}, respectively.}
	\label{filterbunny}
\end{figure*}

\begin{figure*}[h!]
	\centering
	\includegraphics[width=\linewidth]{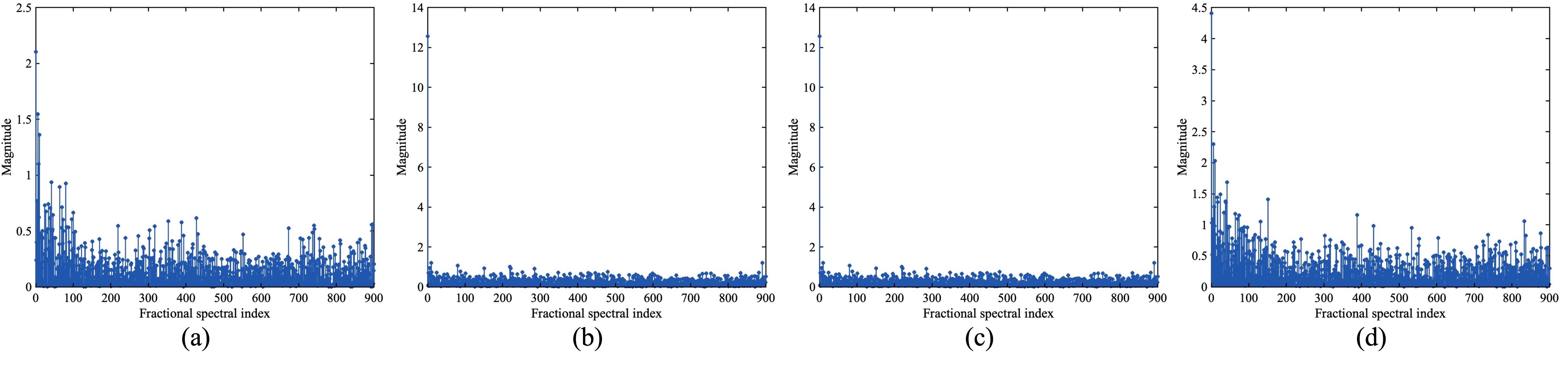}
	\vspace*{-20pt}
	\caption{GFRFT-based spectral representations $|\hat{\bm{g}}|$ of the four spectral filters on the bunny graph. Panels (a)--(d) correspond to \textit{Examples \ref{example1}--\ref{example4}}, respectively.}
	\label{filterhatg}
\end{figure*}

The corresponding uncertainty regions are shown in Fig. \ref{filters}. 
The internal dots represent the localization pairs of the leading vertex-fractional spectral atoms, namely, $(\zeta_f(\bm{\psi}_k),\eta_g^{\alpha}(\bm{\psi}_k))$, where $\bm{\psi}_k$ is a leading eigenvector of the operator $\mathbf{S}_{f,g}^{\alpha}$. The ringed dot corresponds to the largest eigenvalue and therefore marks the dominant joint-localized atom. These points lie inside the admissible region and show how the leading atoms are distributed relative to the uncertainty boundary. Fig. \ref{filtersall} overlays the four uncertainty regions for direct comparison, showing that their distinct shapes arise not only from the graph topology but also from the chosen vertex-domain and fractional spectral filters, which confirms the flexibility of the proposed operator-based framework for different localization definitions.
\begin{figure*}[h!]
	\centering
	\includegraphics[width=\linewidth]{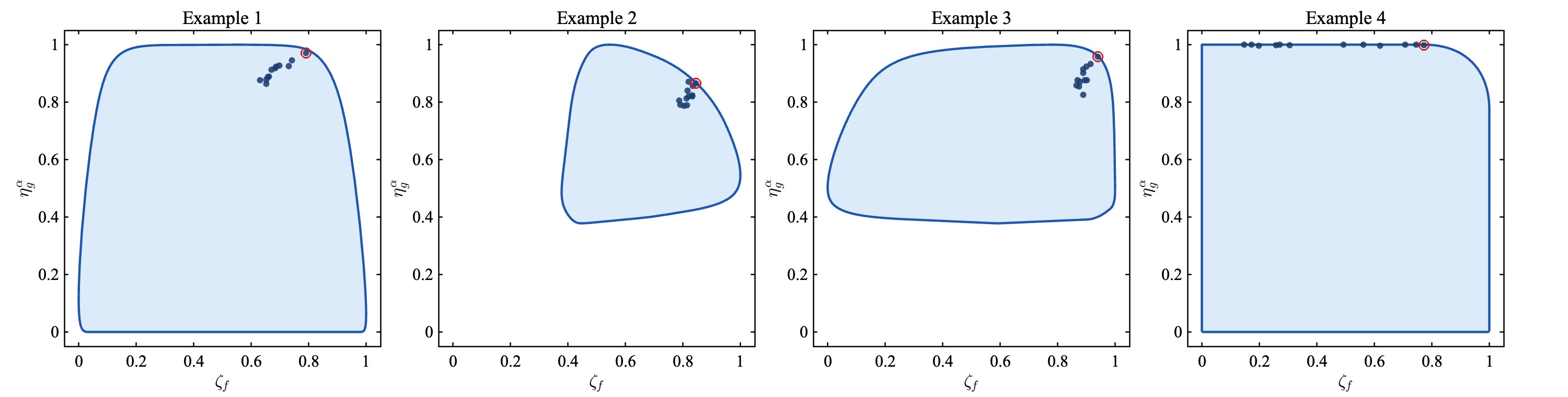}
	\vspace*{-15pt}
	\caption{Uncertainty regions for four localization examples on the bunny graph. The internal dots indicate leading vertex-fractional spectral atoms, and the ringed dot marks the dominant atom.}
	\label{filters}
\end{figure*}

\begin{figure}[h!]
	\centering
	\includegraphics[width=0.55\linewidth]{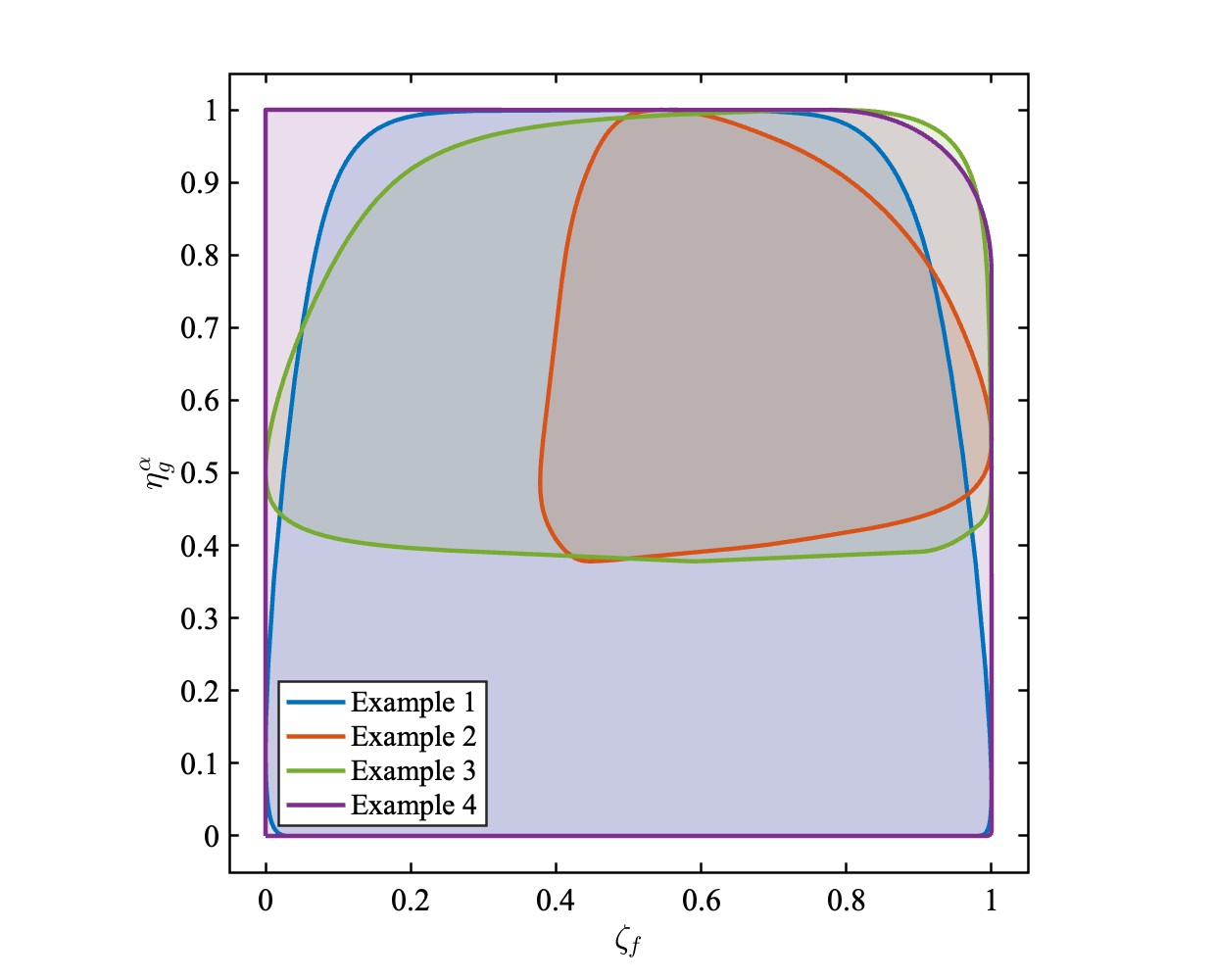}
	\vspace*{-8pt}
	\caption{Overlay of the four uncertainty regions on the bunny graph.}
	\label{filtersall}
\end{figure}

\subsection{Spectral Analysis of Joint Localization Operators}

We further analyze the spectra of the joint localization operators on the bunny graph. This experiment adopts \textit{Example~\ref{example3}} as the representative localization setting. The purpose of this experiment is to quantify how the fractional order and the rotation angle affect the operator spectra underlying the uncertainty regions. We consider the sandwiched joint localization operator $\mathbf{S}_{f,g}^{\alpha}$ and the rotated operator $\mathbf{R}_{f,g}^{\alpha}(\beta)$. In the graph fractional-order scan, the rotation angle is fixed at $\beta=\pi/4$. In the rotation-angle scan, the graph fractional order is fixed at $\alpha=1.0$.

Fig.~\ref{SRbunny} summarizes the eigenvalue decay behavior of the two operators. 
Panel (a) shows the eigenvalue decay of $\mathbf{S}_{f,g}^{\alpha}$ for different fractional orders, reflecting how the fractional representation affects the concentration of joint-localized modes. 
Panel (b) shows the eigenvalue decay of $\mathbf{R}_{f,g}^{\alpha}(\beta)$ for different fractional orders at $\beta=\pi/4$, revealing the influence of the fractional order on the supporting operator associated with the uncertainty boundary.
Panel (c) shows the eigenvalue decay of $\mathbf{R}_{f,g}^{\alpha}(\beta)$ for different rotation angles at fixed fractional order, illustrating how the rotation angle changes the balance between vertex-domain and fractional spectral-domain localization. 
\begin{figure*}[h!]
	\centering
	\includegraphics[width=\linewidth]{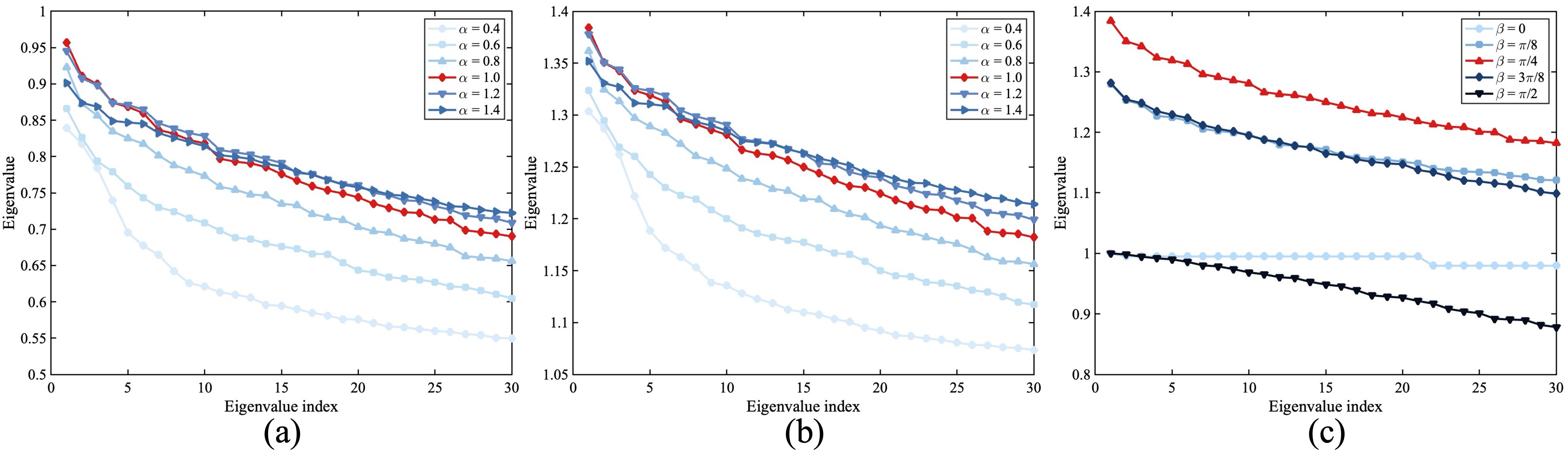}
	\vspace*{-20pt}
	\caption{Spectral analysis of the joint localization operators on the bunny graph. (a) Eigenvalue decay of the sandwiched operator $\mathbf{S}_{f,g}^{\alpha}$ for different fractional orders. (b) Eigenvalue decay of $\mathbf{R}_{f,g}^{\alpha}(\beta)$ for different fractional orders at fixed $\beta=\pi/4$. (c) Eigenvalue decay of the rotated operator $\mathbf{R}_{f,g}^{\alpha}(\beta)$ for different rotation angles at fixed $\alpha=1.0$.}
	\label{SRbunny}
\end{figure*}

To further quantify the spectra, Tables~\ref{tab01} and~\ref{tab02} report several theory-related metrics. For $\mathbf{S}_{f,g}^{\alpha}$, we list the largest eigenvalue $\sigma_1$, the variance $\mathrm{var}[\mathbf{S}_{f,g}^{\alpha}]$, and the corresponding upper bound $B_{\varkappa}$. For $\mathbf{R}_{f,g}^{\alpha}(\beta)$, we report $\rho_1(\beta)$, $\mathrm{var}[\mathbf{R}_{f,g}^{\alpha}]$, and the upper bound $B_{\kappa}$. The representative signal is chosen as the normalized vertex-domain localization filter, namely $\bm{x}=\bm{f}/|\bm{f}|_2$, and the truncation threshold is set to $0.9$ times the largest eigenvalue.
\begin{table}[h!]
	\centering
	\caption{Theory-related metrics for different fractional orders $\alpha$ on the bunny graph}
	\setlength{\tabcolsep}{4pt}
	\renewcommand{\arraystretch}{1.2}
	\begin{tabular}{c|ccc|ccc}
		\hline
		$\alpha$ 
		& $\sigma_1$ 
		& $\mathrm{var}[\mathbf{S}^{\alpha}_{f,g}]$ 
		& $B_{\varkappa}$
		& $\rho_1(\beta)$ 
		& $\mathrm{var}[\mathbf{R}^{\alpha}_{f,g}]$ 
		& $B_{\kappa}$ \\
		\hline
		0.4 & 0.84 & 0.03 & 4.08 & 1.30 & 0.03 & 2.26 \\
		0.6 & 0.87 & 0.05 & 2.80 & 1.32 & 0.04 & 1.56 \\
		0.8 & 0.92 & 0.04 & 1.75 & 1.36 & 0.03 & 0.92 \\
		1.0 & 0.96 & 0.02 & 1.39 & 1.38 & 0.01 & 0.68 \\
		1.2 & 0.95 & 0.04 & 1.93 & 1.38 & 0.03 & 1.03 \\
		1.4 & 0.90 & 0.06 & 3.16 & 1.35 & 0.05 & 1.77 \\
		\hline
	\end{tabular}
	\label{tab01}
\end{table}

\begin{table}[h!]
	\centering
	\caption{Theory-related metrics for different rotation angles $\beta$ on the bunny graph}
	\setlength{\tabcolsep}{5pt}
	\renewcommand{\arraystretch}{1.2}
	\begin{tabular}{c|ccc}
		\hline
		$\beta$ 
		& $\rho_1(\beta)$ 
		& $\mathrm{var}[\mathbf{R}^{\alpha}_{f,g}]$ 
		& $B_{\kappa}$ \\
		\hline
		$0$       & 1.00 & 0.02 & 1.74 \\
		$\pi/8$  & 1.28 & 0.02 & 1.06 \\
		$\pi/4$  & 1.38 & 0.01 & 0.68 \\
		$3\pi/8$  & 1.28 & 0.00 & 0.34 \\
		$\pi/2$  & 1.00 & 0.00 & 0.02 \\
		\hline
	\end{tabular}
	\label{tab02}
\end{table}

The results in Table~\ref{tab01} show that the fractional order affects both the dominant eigenvalues and the approximation-related quantities of the joint localization operators. Variations in $\sigma_1$ and $\rho_1(\beta)$ indicate that the leading joint-localized modes depend on the fractional representation, while the changes in the variances and bounds reflect different concentration behaviors with respect to the associated eigenspaces. Table~\ref{tab02} further shows that the rotation angle changes the spectral behavior of $\mathbf{R}_{f,g}^{\alpha}(\beta)$ by adjusting the relative weights of vertex-domain and fractional spectral-domain localization. These observations confirm that the fractional order and rotation angle provide complementary degrees of freedom for characterizing joint localization and the boundary geometry of the uncertainty region.

\section{Conclusion}
\label{sec6}
In this paper, we proposed a graph fractional uncertainty principle based on the GFRFT, which extends classical graph uncertainty relations from the graph Fourier domain to more general fractional spectral domains. The framework allows flexible characterization and visualization of how graph signals are localized across both vertex and spectral domains. Our results show that classical uncertainty relations appear as a special case, while other fractional settings reveal richer localization patterns. This approach provides a practical and theoretical tool for analyzing graph signals and designing graph-adapted filters.

\appendices
\section{Proof of Theorem 1}
\label{appendixA}
We prove the result for each operator separately.

By definition, $\mathbf{D}_{f}=\mathrm{diag}(\bm{f})$, where $0 \le f_i \le 1$ and $\|\bm{f}\|_{\infty}=1$. Hence, $\mathbf{D}_{f}$ is a real diagonal matrix, which is clearly Hermitian. Moreover, for any $\bm{x}\in\mathbb{C}^N$,
\[
\bm{x}^{\mathrm{H}}\mathbf{D}_{f}\bm{x}
=
\sum_{i=1}^{N} f_i |x_i|^2 \ge 0,
\]
which shows that $\mathbf{D}_{f}$ is positive semidefinite. Since its eigenvalues are $\{f_i\}$, its spectral norm satisfies
\[
\|\mathbf{D}_{f}\|_2 = \max_i f_i = \|\bm{f}\|_{\infty} = 1.
\]

The operator $\mathbf{B}^{\alpha}_{g}$ is defined as
\[
\mathbf{B}^{\alpha}_{g}
=\mathbf{F}^{-\alpha} \mathbf{D}_{\hat{g}} \mathbf{F}^{\alpha}=
(\mathbf{F}^{\alpha})^{\mathrm{H}} \mathbf{D}_{\hat{g}} \mathbf{F}^{\alpha},
\]
where $\mathbf{F}^{\alpha}$ is unitary and $\mathbf{D}_{\hat{g}}=\mathrm{diag}(\hat{g})$ with $0 \le \hat{g}_i \le 1$ and $\|\hat{g}\|_{\infty}=1$. Hence,
\[
(\mathbf{B}^{\alpha}_{g})^{\mathrm{H}}
=
(\mathbf{F}^{\alpha})^{\mathrm{H}} \mathbf{D}_{\hat{g}} \mathbf{F}^{\alpha}
=
\mathbf{B}^{\alpha}_{g},
\]
which shows that $\mathbf{B}^{\alpha}_{g}$ is Hermitian. For any $\bm{x}\in\mathbb{C}^N$,
\[
\bm{x}^{\mathrm{H}}\mathbf{B}^{\alpha}_{g}\bm{x}
=
(\mathbf{F}^{\alpha}\bm{x})^{\mathrm{H}} \mathbf{D}_{\hat{g}} (\mathbf{F}^{\alpha}\bm{x})
=
\sum_{i=1}^{N} \hat{g}_i |\hat{x}_i|^2 \ge 0,
\]
where $\hat{\bm{x}}=\mathbf{F}^{\alpha}\bm{x}$ is defined in Eq. \eqref{GFRFT}. Thus, $\mathbf{B}^{\alpha}_{g}$ is positive semidefinite. 

Since $\mathbf{B}^{\alpha}_{g}$ is unitarily similar to $\mathbf{D}_{\hat{g}}$, they share the same eigenvalues. Therefore,
\[
\|\mathbf{B}^{\alpha}_{g}\|_2
=
\|\mathbf{D}_{\hat{g}}\|_2
=
\max_i \hat{g}_i
=
\|\hat{g}\|_{\infty}
=
1.
\]

This completes the proof.

\section{Proof of Theorem 2}
\label{appendixB}
We first show that $\mathbf{S}^{\alpha}_{f,g}$ is Hermitian. Since both 
$\mathbf{D}_f$ and $(\mathbf{B}^{\alpha}_{g})^{\frac{1}{2}}$ are Hermitian, we have
\[
(\mathbf{S}^{\alpha}_{f,g})^{\mathrm{H}}
=
(\mathbf{B}^{\alpha}_{g})^{\frac{1}{2}}
\mathbf{D}_{f}
(\mathbf{B}^{\alpha}_{g})^{\frac{1}{2}}
=
\mathbf{S}^{\alpha}_{f,g},
\]
which proves that $\mathbf{S}^{\alpha}_{f,g}$ is Hermitian.

Next, for any $\bm{x}\in\mathbb{C}^N$, define 
$\bm{y} = (\mathbf{B}^{\alpha}_{g})^{\frac{1}{2}}\bm{x}$. Then
\[
\bm{x}^{\mathrm{H}}\mathbf{S}^{\alpha}_{f,g}\bm{x}
=
\bm{y}^{\mathrm{H}}\mathbf{D}_f \bm{y}
=
\sum_{i=1}^{N} f_i |y_i|^2 \ge 0,
\]
since $f_i \ge 0$. Hence, $\mathbf{S}^{\alpha}_{f,g}$ is positive semidefinite.

Finally, we bound its spectral norm. Using the submultiplicativity of the operator norm,
\[
\|\mathbf{S}^{\alpha}_{f,g}\|_2
\le
\|(\mathbf{B}^{\alpha}_{g})^{\frac{1}{2}}\|_2^2 \|\mathbf{D}_f\|_2.
\]
Since $(\mathbf{B}^{\alpha}_{g})^{\frac{1}{2}}$ is positive semidefinite,
\[
\|(\mathbf{B}^{\alpha}_{g})^{\frac{1}{2}}\|_2^2
=
\|\mathbf{B}^{\alpha}_{g}\|_2,
\]
and from \textit{Theorem 1} we have $\|\mathbf{B}^{\alpha}_{g}\|_2 = 1$, and $\|\mathbf{D}_f\|_2 = 1$. Therefore,
\[
\|\mathbf{S}^{\alpha}_{f,g}\|_2 \le 1.
\]

This completes the proof.

\section{Proof of Lemma 1}
\label{appendixC}

Let $\bm{x}\in\mathbb{C}^N$ with $\|\bm{x}\|_2=1$, and define the normalized vectors
\[
\bm{y}=\frac{\mathbf{D}_f \bm{x}}{\|\mathbf{D}_f \bm{x}\|_2}, 
\quad
\bm{z}=\frac{\mathbf{B}^{\alpha}_g \bm{x}}{\|\mathbf{B}^{\alpha}_g \bm{x}\|_2}.
\]
By construction, $\|\bm{y}\|_2=\|\bm{z}\|_2=1$. Denote the angular distance between two unit vectors by
\[
\angle(\bm{y},\bm{z})=\arccos\big(\Re\langle \bm{y}, \bm{z} \rangle\big).
\]
Using the triangle inequality on the unit sphere, we have
\begin{equation}
	\angle(\bm{y},\bm{z}) \le \angle(\bm{y},\bm{x}) + \angle(\bm{z},\bm{x}).
	\label{eq:angle_triangle}
\end{equation}

We next bound the correlation between $\bm{y}$ and $\bm{z}$. By Cauchy–Schwarz inequality,
\[
\Re\langle \bm{y},\bm{z}\rangle 
\le |\langle \bm{y},\bm{z}\rangle|
= \frac{|\langle \mathbf{D}_f \bm{x}, \mathbf{B}^{\alpha}_g \bm{x}\rangle|}
{\|\mathbf{D}_f \bm{x}\|_2 \, \|\mathbf{B}^{\alpha}_g \bm{x}\|_2}.
\]
Using the definition of the operator $\mathbf{S}^{\alpha}_{f,g}
= (\mathbf{B}^{\alpha}_g)^{\frac{1}{2}} \mathbf{D}_f (\mathbf{B}^{\alpha}_g)^{\frac{1}{2}}$, we obtain
\[
\begin{aligned}
	|\langle \mathbf{D}_f \bm{x}, \mathbf{B}^{\alpha}_g \bm{x}\rangle|
	= &
	|\langle \mathbf{D}^{\frac{1}{2}}_f \bm{x}, \mathbf{D}^{\frac{1}{2}}_f\mathbf{B}^{\alpha}_g \bm{x}\rangle|
	\leq  \| \mathbf{D}^{\frac{1}{2}}_f \bm{x} \|_2 \| \mathbf{D}^{\frac{1}{2}}_f \mathbf{B}^{\alpha}_g \bm{x}\|_2 \\
	= & \| \mathbf{D}^{\frac{1}{2}}_f \bm{x} \|_2 \sqrt{\left< (\mathbf{B}_g^{\alpha})^{\frac{1}{2}}\mathbf{D}_f\mathbf{B}_g^{\alpha}\bm{x},(\mathbf{B}_g^{\alpha})^{\frac{1}{2}}\bm{x}\right> } \\
	\leq & \| \mathbf{D}^{\frac{1}{2}}_f \bm{x} \|_2  \| (\mathbf{B}_g^{\alpha})^{\frac{1}{2}}\bm{x}\|_2 \sqrt{\left< \mathbf{S}^{\alpha}_{f,g}\bm{x}, \bm{x}\right> }\\
	\leq & \| \mathbf{D}_f \bm{x} \|_2  \| \mathbf{B}_g^{\alpha}\bm{x}\|_2 \sqrt{\sigma_1},
\end{aligned}
\]
where $\sigma_1=\|\mathbf{S}^{\alpha}_{f,g}\|_2$ is the largest eigenvalue. Hence,
\[
\Re\langle \bm{y},\bm{z}\rangle \le \sqrt{\sigma_1}.
\]

On the other hand, by definition,
\[
\Re\langle \bm{y},\bm{x}\rangle = \sqrt{\zeta_f(\bm{x})},
\quad
\Re\langle \bm{z},\bm{x}\rangle = \sqrt{\eta^{\alpha}_g(\bm{x})}.
\]
Substituting these relations into \eqref{eq:angle_triangle} yields
\[
\arccos \sqrt{\zeta_f(\bm{x})}
+
\arccos \sqrt{\eta^{\alpha}_g(\bm{x})}
\ge
\arccos \sqrt{\sigma_1},
\]
which proves \eqref{arccoszeta_eta}. To derive \eqref{etaalpha}, note that $\arccos(\cdot)$ is monotonically decreasing on $(0,1]$, hence
\[
\arccos \sqrt{\eta^{\alpha}_g(\bm{x})}
\ge
\arccos \sqrt{\sigma_1}
-
\arccos \sqrt{\zeta_f(\bm{x})}.
\]
Applying the cosine function and using $\cos(a-b)=\cos a \cos b + \sin a \sin b$, we obtain
\[
\sqrt{\eta^{\alpha}_g(\bm{x})}
\le
\sqrt{\sigma_1 \zeta_f(\bm{x})}
+
\sqrt{(1-\sigma_1)(1-\zeta_f(\bm{x}))},
\]
which leads to
\[
\eta^{\alpha}_g(\bm{x})
\le
\left(
\sqrt{\sigma_1 \zeta_f(\bm{x})}
+
\sqrt{(1-\sigma_1)(1-\zeta_f(\bm{x}))}
\right)^2.
\]
This completes the proof.

\section{Proof of Theorem 4}
\label{appendixD}
For any $\beta \in [0,2\pi)$, consider the rotated joint operator
\[
\mathbf{R}^{\alpha}_{f,g}(\beta)
=
\cos\beta\, \mathbf{D}_f + \sin\beta\, \mathbf{B}_g^{\alpha}.
\]
Then, the linear combination of the localization measures can be written as
\[
\cos\beta\, \zeta_f(\bm{x}) + \sin\beta\, \eta_g^{\alpha}(\bm{x})
=
\bm{x}^{\mathrm{H}} \mathbf{R}^{\alpha}_{f,g}(\beta)\bm{x}.
\]
Since $\mathbf{R}^{\alpha}_{f,g}(\beta)$ is Hermitian, its Rayleigh quotient is bounded above by its largest eigenvalue $\rho_1(\beta)$, namely,
\[
\bm{x}^{\mathrm{H}} \mathbf{R}^{\alpha}_{f,g}(\beta)\bm{x}
\le \rho_1(\beta).
\]
Moreover, equality is attained when $\bm{x}$ is chosen as an eigenvector corresponding to $\rho_1(\beta)$, implying that $\ell(\beta)$ is a supporting line of the admissible region.

Therefore, for any admissible pair $(\zeta_f(\bm{x}), \eta_g^{\alpha}(\bm{x}))$, we obtain
\[
\cos\beta\, \zeta_f(\bm{x}) + \sin\beta\, \eta_g^{\alpha}(\bm{x})
\le \rho_1(\beta).
\]

Since $\bm{x}$ is arbitrary with $\|\bm{x}\|_2=1$, the above inequality holds for all points in $\Gamma(\mathbf{D}_f,\mathbf{B}_g^{\alpha})$, which completes the proof.

\section*{Acknowledgments}
The authors would like to sincerely thank Prof. Wolfgang Erb for generously sharing his code.


\begin{thebibliography}{99}
\bibliographystyle{IEEEtran}

\bibitem{GFTlaplace} 
D.~I.~Shuman, S.~K.~Narang, P.~Frossard, A.~Ortega, and P.~Vandergheynst, ``The emerging field of signal processing on graphs: Extending high-dimensional data analysis to networks and other irregular domains,'' \textit{IEEE Signal Process. Mag.}, vol. 30, no. 3, pp. 83--98, 2013.

\bibitem{GFTadjacency} 
A.~Sandryhaila and J.~M.~F.~Moura, ``Big data processing with signal processing on graphs: Representation and processing of massive data sets with irregular structure,'' \textit{IEEE Signal Process. Mag.}, vol. 31, no. 5, pp. 80--90, 2014.

\bibitem{Gintroduction}
A.~Ortega, \textit{Introduction to Graph Signal Processing}, Cambridge University Press, Cambridge, U.K., 2022.

\bibitem{Goverview} 
A.~Ortega, P.~Frossard, J.~Kova\v{c}evi\'{c}, J.~M.~F.~Moura, and P.~Vandergheynst, ``Graph signal processing: Overview, challenges, and applications,'' \textit{Proc. IEEE}, vol. 106, no. 5, pp. 808--828, 2018.			

\bibitem{Ghistory} 
G.~Leus, A.~G.~Marques, J.~M.~F.~Moura, A.~Ortega, and D.~I.~Shuman, ``Graph signal processing: History, development, impact, and outlook,'' \textit{IEEE Signal Process. Mag.}, vol. 40, no. 4, pp. 49--60, 2023.

\bibitem{UncertaintyC}
G.~B.~Folland, A.~Sitaram, ``The uncertainty principle: A mathematical survey,'' \textit{J. Fourier Anal. Appl.}, vol. 3, no. 3, pp. 207--238, 1997.

\bibitem{UncertaintyM}
D.~Slepian, ``Some comments on Fourier analysis, uncertainty and modeling,'' \textit{SIAM Rev.}, vol. 25, no. 3, pp. 379--393, 1983.

\bibitem{GspectralUC} %graph uncertainty principle
A.~Agaskar, Y.~M.~Lu, ``A spectral graph uncertainty principle,'' \textit{IEEE Trans. Inf. Theory}, vol. 59, no. 7, pp. 4338--4356, 2013.

\bibitem{Resistance}
D.~Klein and M.~Randi\'{c}, ``Resistance distance,'' \textit{J. Math. Chem.}, vol. 12, no. 1, pp. 81--95, 1993.

\bibitem{Geodesic}
C.~Godsil and G.~Royle, \textit{Algebraic Graph Theory}. NewYork, NY, USA: Springer-Verlag, 2001.

\bibitem{Diffusion}
R.~R.~Coifman and M.~Maggioni, ``Diffusion wavelets,'' \textit{Appl. Comput. Harmon. Anal.}, vol. 21, no. 1, pp. 53--94, 2006.

\bibitem{GUncertainty} %
M.~Tsitsvero, S.~Barbarossa, P.~D.~Lorenzo, ``Signals on graphs: Uncertainty principle and sampling,'' \textit{IEEE Trans. Signal Process.}, vol. 64, no. 18. pp. 4845--4860, 2016.

\bibitem{PSWfunctions1}
D.~Slepian, H.~O.~Pollak, ``Prolate spheroidal wave functions, Fourier analysis and uncertainty. I,'' \textit{Bell Syst. Tech. J.}, vol. 40, no. 1, pp. 43--63, 1961.

\bibitem{PSWfunctions2}
H.~J.~Landau, H.~O.~Pollak, ``Prolate spheroidal wave functions, Fourier analysis and uncertainty. II,'' \textit{Bell Syst. Tech. J.}, vol. 40, no. 1, pp. 65--84, 1961.

\bibitem{GshapesUC} 
W.~Erb, ``Shapes of uncertainty in spectral graph theory,'' \textit{IEEE Trans. Inf. Theory}, vol. 67, no. 2, pp. 1291--1307, 2021.

\bibitem{GLCTsampling} 
Y.~Zhang and B.~Z.~Li, ``Discrete linear canonical transform on graphs: Uncertainty principle and sampling,'' \textit{Signal Process.}, vol. 226, 2025, Art. no. 109668.

\bibitem{JFTUC}
Y.~Zhao, X.~Jian, F.~Ji, W.~P.~Tay, and A.~Ortega, ``Uncertainty principle for vertex-time graph signal processing,'' 2026, arXiv preprint, arXiv:2602.04084.

\bibitem{GGSPUC} 
Y.~Zhao, X.~Jian, F.~Ji, W.~P.~Tay, and A.~Ortega, ``Generalized graph signal reconstruction via the uncertainty principle," in \textit{Proc. IEEE Int. Conf. Acoust., Speech, Signal Process. (ICASSP)}, 2025, pp. 1--5.

\bibitem{GFRFT} 
Y.~Q.~Wang, B.~Z.~Li, and Q.~Y.~Cheng, ``The fractional Fourier transform on graphs,'' in
\textit{Proc. Asia-Pacific Signal Inf. Process. Assoc. Annu. Summit Conf. (APSIPA ASC)}, 2017, pp. 105--110.

\bibitem{GFRFTconvolution} 
Y.~Zhang and B.~Z.~Li, ``The fractional Fourier transform on graphs: Modulation and convolution,'' in \textit{Proc. Int. Conf. Signal Image Process. (ICSIP)}, 2023, pp. 737-741.

\bibitem{GFRFTspectral} 
J.~Wu, F.~Wu, Q.~Yang, Y.~Zhang, X.~Liu, Y.~Kong, L.~Senhadji, and H.~Shu, ``Fractional spectral graph wavelets and their applications,'' \textit{Math. Probl. Eng.}, 2020.

\bibitem{FRFT1}
L.~B.~Almeida, ``The fractional Fourier transform and time-frequency representations,'' \textit{IEEE Trans. Signal Process.}, vol. 42, no. 11, pp. 3084--3091, 1994.

\bibitem{FRFT2}
H.~M.~Ozaktas, Z.~Zalevsky, and M.~A.~Kutay, \textit{The Fractional Fourier Transform with Applications in Optics and Signal Processing}. New York, NY, USA: Wiley, 2001.

\bibitem{DFRFT}
C.~Candan, M.~A.~Kutay, and H.~M.~Ozaktas, ``The discrete fractional Fourier transform,'' \textit{IEEE Trans. Signal Process.}, vol. 48, no. 5, pp. 1329--1337, 2000.

\bibitem{GFRFT_unified}
T.~Alika{\c{s}}ifo{\u{g}}lu, B.~Kartal, and A.~Ko{\c{c}}, ``Graph fractional Fourier transform: A unified theory,'' \textit{IEEE Trans. Signal Process.}, vol. 72, pp. 3834--3850, 2024.

\bibitem{GFRFTdirected} 
F.~J.~Yan and B.~Z.~Li, ``Spectral graph fractional Fourier transform for directed graphs and its application,'' \textit{Signal Process.}, vol. 210, 2023, Art. no. 109099. 

\bibitem{WGFRFT} 
F.~J.~Yan and B.~Z.~Li, ``Windowed fractional Fourier transform on graphs: Properties and fast algorithm,'' \textit{Digit. Signal Process.}, vol. 118, 2021, Art. no. 103210.

\bibitem{JFRFT}
T.~Alika{\c{s}}ifo{\u{g}}lu, B.~Kartal, E.~\"{O}zg\"{u}nay, and A.~Ko{\c{c}}, ``Joint time-vertex fractional Fourier transform,'' \textit{Signal Process.}, vol. 233, 2025, Art. no. 109944.

\bibitem{GFRFTfiltering} 
C.~Ozturk, H.~M.~Ozaktas, S.~Gezici, and A.~Koç, ``Optimal fractional Fourier filtering for graph signals,'' \textit{IEEE Trans. Signal Process.}, vol. 69, pp. 2902--2912, 2021.

\bibitem{JFRFTfilter} 
Z.~Ge, H.~Guo, T.~Wang, and Z.~Yang, ``The optimal joint time-vertex graph filter design: From ordinary graph Fourier domains to fractional graph Fourier domains,'' \textit{Circuits, Syst. Signal Process.}, vol. 42, pp. 4002--4018, 2023.

\bibitem{JFRFTwiener}
T.~Alika\c{s}\i fo\u{g}lu, B.~Kartal, and A.~Ko\c{c}, ``Wiener filtering in joint time-vertex fractional Fourier domains,'' \textit{IEEE Signal Process. Lett.}, vol. 31, pp. 1319--1323, 2024.

\bibitem{HGFRFT} 
Y.~Zhang and B.~Z.~Li, ``The graph fractional Fourier transform in Hilbert space,'' \textit{IEEE Trans. Signal Inf. Process. Netw.}, vol. 11, pp. 242--257, 2025.

\bibitem{GFRFTsampling} 
Y.~Q.~Wang and B.~Z.~Li, ``The fractional Fourier transform on graphs: Sampling and recovery,'' in \textit{Proc. 14th IEEE Int. Conf. Signal Process. (ICSP)}, 2018, pp. 1103--1108.

\bibitem{JFRFTsampling}
Y.~Zhang and B.~Z.~Li, ``Sampling of graph signals based on joint time-vertex fractional Fourier transform,'' \textit{Signal Process.}, vol. 239, 2026, Art. no. 110309.

\bibitem{LSBreconstruction}
X.~Wang, P.~Liu, and Y.~Gu, ``Local-set-based graph signal reconstruction,'' \textit{IEEE Trans. Signal Process.}, vol. 63, no. 9, pp. 2432--2444, 2015.

\bibitem{GGSPreconstruction}
X.~Jian, W.~P.~Tay, and Y.~C.~Eldar, ``Kernel based reconstruction for generalized graph signal processing,'' \textit{IEEE Trans. Signal Process.}, vol. 72, pp. 2308--2322, 2024.

\bibitem{ComplexKernel}
Y.~Zhang, L.~Peng, and B.~Z.~Li, ``Reconstruction of graph signals on complex manifolds with kernel methods,'' \textit{IEEE Trans. Signal Process.}, 2026.

\bibitem{Optlocalized}
W.~Erb, ``Optimally space localized polynomials with applications in signal processing,'' \textit{J. Fourier Anal. Appl.}, vol. 18, no. 1, pp. 45--66, 2012.

\bibitem{Slepian}
W.~Erb and S.~Mathias, ``An alternative to Slepian functions on the unit sphere – A space–frequency analysis based on localized spherical polynomials,'' \textit{Appl. Comput. Harmon. Anal.}, vol. 38, no. 2, pp. 222--241, 2015.

\bibitem{Erb}
H.~Klaja, ``On Erb's uncertainty principle,'' \textit{Studia Math.}, vol. 232, no. 1, pp. 7--17, 2016.

\bibitem{Fiedler}
D.~Van De Ville, R.~Demesmaeker, and M.~G.~Preti, ``
When Slepian meets Fiedler: Putting a focus on the graph spectrum,'' \textit{IEEE Signal Process. Lett.}, vol. 24, no. 7, pp. 1001--1004, 2017.

\bibitem{Spatio1}
F.~J.~Simons, F.~A.~Dahlen, and M.~A.~Wieczorek, ``Spatiospectral concentration on a sphere,'' \textit{SIAM Rev.}, vol. 48, no. 3, pp. 504--536, 2006.

\bibitem{Spatio2}
A. Plattner and F. J. Simons, ``Spatiospectral concentration of vector fields on a sphere,'' \textit{Appl. Comput. Harmon. Anal.}, vol. 36, no. 1, pp. 1--22, 2014.

\bibitem{GraphTheory}
F.~R.~K.~Chung, \textit{Spectral Graph Theory}. Providence, RI, USA: American Mathematical Society, 1997.

\bibitem{Hausdorff}
R.~A.~Horn and C.~R.~Johnson, \textit{Topics in Matrix Analysis}. Cambridge, U.K.: Cambridge University Press, 1991.

\bibitem{GSPBox}
N.~Perraudin, J.~Paratte, D.~Shuman, V.~Kalofolias, P.~Vandergheynst, D.~K.~Hammond, GSPBOX: a toolbox for signal processing on graphs, 2016, arXiv preprint, arXiv: 1408.5781.





\end{thebibliography}
\end{document}